\documentclass{amsart}

\usepackage[margin=1.2in]{geometry}
\usepackage{amsmath}
\usepackage{latexsym}
\usepackage{amsfonts}
\usepackage{amssymb}
\usepackage{amsmath}
\usepackage{amsthm}
\usepackage{amsxtra}
\usepackage{amscd}
\usepackage[foot]{amsaddr}
\usepackage{a4wide}
\usepackage{color}
\usepackage[utf8]{inputenc}
\usepackage{todonotes}
\usepackage{mathrsfs}
\usepackage{hyperref}
\usepackage{cite}

\allowdisplaybreaks

\newtheorem{theorem}{Theorem}[section]
\newtheorem{lemma}[theorem]{Lemma}
\newtheorem{proposition}[theorem]{Proposition}
\newtheorem{corollary}[theorem]{Corollary}

\theoremstyle{definition}
	\newtheorem{definition}[theorem]{Definition}
	\newtheorem{example}[theorem]{Example}
	
	\newtheorem{remark}[theorem]{Remark}

\numberwithin{equation}{section}

\newcommand{\N}{\mathbb{N}} 
\newcommand{\R}{\mathbb{R}} 
\newcommand{\C}{\mathbb{C}} 

\newcommand{\Fmult}{F_{\alpha,\lambda}^{\varphi,\psi}}

\newcommand{\Fmulti}{F_{\alpha,\lambda}^{\varphi_i,\psi_i}}

\newcommand{\Fmultj}{F_{\alpha,\lambda}^{\varphi_{i_j},\psi_{i_j}}}

\newcommand{\Fmultoj}{F_{\alpha,{\lambda_0}}^{\varphi_{i_j},\psi_{i_j}}}

\newcommand{\Fmultiterate}{F_{\alpha,\lambda}^{\varphi_n,\psi^{n,\varphi}}}

\begin{document}
	
\title[Weighted composition operators on $\mathscr{S}(\R^d)$]{Power boundedness and related properties for weighted composition operators on $\mathscr{S}(\R^d)$}

\author[V.\ Asensio]{Vicente Asensio$^1$}
\address{$^1$Instituto Universitario de Matem\'atica Pura y Aplicada IUMPA, Universitat Polit\`ecnica de Val\`encia, Camino de Vera, s/n, E-46022 Valencia, Spain}
\email{viaslo@upv.es}

\author[E.\ Jord\'a]{Enrique Jord\'a$^2$}
\address{$^2$Departamento de Matem\'atica Aplicada, E. Polit\'{e}cnica Superior
	de Alcoy, Universidad Polit\'ecnica de Valencia, Plaza Ferr\'andiz
	y Carbonell 2, E-03801 Alcoy (Alicante), Spain}
\email{ejorda@mat.upv.es}

\author[T.\ Kalmes]{Thomas Kalmes$^3$}
\address{$^3$Faculty of Mathematics, Chemnitz University of Technology, 09107 Chemnitz, Germany}
\email{thomas.kalmes@math.tu-chemnitz.de}

\date{}

\maketitle

\begin{abstract}
	We characterize those pairs $(\psi,\varphi)$ of smooth mappings $\psi:\mathbb{R}^d\rightarrow\mathbb{C}, \varphi:\mathbb{R}^d\rightarrow\mathbb{R}^d$ for which the corresponding weighted composition operator $C_{\psi,\varphi}f=\psi\cdot(f\circ\varphi)$ acts continuously on $\mathscr{S}(\mathbb{R}^d)$. Additionally, we give several easy-to-check necessary and sufficient conditions of this property for interesting special cases. Moreover, we characterize power boundedness and topologizablity of $C_{\psi,\varphi}$ on $\mathscr{S}(\mathbb{R}^d)$ in terms of $\psi,\varphi$. Among other things, as an application of our results we show that for a univariate polynomial $\varphi$ with $\text{deg}(\varphi)\geq 2$, power boundedness of $C_{\psi,\varphi}$ on $\mathscr{S}(\mathbb{R})$ for every $\psi\in\mathscr{O}_M(\mathbb{R})$ only depends on $\varphi$ and that in this case power boundedness of $C_{\psi,\varphi}$ is equivalent to $(C_{\psi,\varphi}^n)_{n\in\mathbb{N}}$ converging to $0$ in $\mathcal{L}_b(\mathscr{S}(\mathbb{R}))$ as well as to the uniform mean ergodicity of $C_{\psi,\varphi}$. Additionally, we give an example of a power bounded and uniformly mean ergodic weighted composition operator $C_{\psi,\varphi}$ on $\mathscr{S}(\mathbb{R})$ for which neither the multiplication operator $f\mapsto \psi f$ nor the composition operator $f\mapsto f\circ\varphi$ acts on $\mathscr{S}(\mathbb{R})$. Our results complement and considerably extend various results of Fern\'andez, Galbis, and the second named author. 
	
	\mbox{}\\
	
	
\end{abstract}
	
\section{Introduction}	

Weighted composition operators play an important role in functional analysis and operator theory. Given a set $\Omega$, a self mapping $\varphi$ on $\Omega$ and a function $\psi$ on $\Omega$, beyond the fundamental question of when a weighted composition operator $C_{\psi,\varphi}$ acts on a given function space $\mathcal{F}(\Omega)$ -- that is, when the operation $C_{\psi,\varphi}f=\psi\cdot(f\circ\varphi)$ results in a function belonging to $\mathcal{F}(\Omega)$ for every $f\in\mathcal{F}(\Omega)$ -- it is a natural task to characterize operator theoretic properties of $C_{\psi,\varphi}$ on $\mathcal{F}(\Omega)$ by properties of $\psi$ and $\varphi$. Obviously, the class of weighted composition operators contains multiplication operators, i.e.\ $\varphi(x)=x$, as well as composition operators, i.e.\ $\psi(x)=1$. 

In the present article, we consider weighted composition operators on the space $\mathscr{S}(\R^d)$ of rapidly decreasing smooth functions. While the space $\mathscr{O}_M(\R^d)$ of multipliers for $\mathscr{S}(\R^d)$ has been characterized by L.~Schwartz \cite{Horvath,Schwartz}, the functions $\varphi$ on $\R$ for which the corresponding composition operator acts on $\mathscr{S}(\R)$ have been characterized only recently in \cite{GaJo18}. Apart from characterizing the pairs $(\psi,\varphi)$ for which $C_{\psi,\varphi}$ acts on $\mathscr{S}(\R^d)$, we also characterize power boundedness and ($m$-)topologizability for weighted composition operators on $\mathscr{S}(\R^d)$ in terms of $\psi$ and $\varphi$.

In recent years, mean ergodicity, power boundedness, and topologizablity of (weighted) composition operators and multiplication operators on various spaces of (generalized) functions have attracted the attention of a large number of authors. We give only a sample of articles (and refer to references therein); see e.g.\ \cite{AlJoMe22, BeGoJoJo16b, BeGoJoJo16, BeJo21, BoDo11, BoDo11b, BoJoRo18, BoRi09, GoJoJo16, JoRo20, JoSaSe20, JoSaSe21, Kalmes19, Kalmes20, KaSa22, KaSa23, Santacreu24}.

Recall that a continuous linear operator $T$ on a locally convex Hausdorff  space $E$ is power bounded precisely when the set of its iterates $\{T^n;\,n\in\N\}$ is equicontinuous. This notion is closely connected with $T$ being mean ergodic, i.e.\ with the property that for every $x\in E$ the sequence of Ces\`aro means $\left(\frac{1}{n}\sum_{m=1}^n T^m(x)\right)_{n\in\N}$ converges. Whenever the Ces\`aro means converge uniformly on bounded sets, $T$ is uniformly mean ergodic. By a classical result of Lorch \cite{Lorch39}, on reflexive Banach spaces, every power bounded operator is mean ergodic, which characterizes reflexivity of Banach spaces, as has been shown in the celebrated work \cite{FoLiWo01} by Fonf, Lin, and Wojtaszczyk. Bonet, de Pagter, and Ricker \cite[Proposition 3.3]{BoPaRi11} proved that Lorch's result remains true for (semi-)reflexive Hausdorff locally convex spaces. Additionally, by \cite[Theorem 2.5]{KaSa22}, on Montel spaces, mean ergodic operators are automatically uniformly mean ergodic. Thus, power bounded operators of $\mathscr{S}(\R^d)$ are already uniformly mean ergodic.

While the interest for power boundedness for operators on locally convex Hausdorff spaces stems from its close relationship to (uniform) mean ergodicity, topologizable operators were introduced by \.{Z}elazko in \cite{Zelazko07} (see also \cite{Bonet10}). Recall that a continuous linear operator $T$ on a locally convex Hausdorff space $E$ is topologizable if for every continuous seminorm $p$ on $E$ there is a continuous seminorm $q$ on $E$ and a sequence $(a_n)_{n\in\N}$ of positive numbers such that $(a_n T^n)_{n\in\N}$ is equicontinuous from $E$ equipped with $p$ into $E$ equipped with $q$. This property characterizes those $T$ for which there is a unital subalgebra $A$ of $\mathcal{L}(E)$ (with composition as multiplication) which contains $T$ and which admits a locally convex topology making $A$ into a topological algebra such that the map $A\times E\rightarrow E, (S,x)\mapsto Sx$ is continuous. While the notion of $m$-topologizability (where the sequence $(a_n)_{n\in\N}$ in the definition of topologizablity can be chosen as a sequence of powers $(M^n)_{n\in\N}, M>0$) was also introduced by \.{Z}elazko \cite{Zelazko02}, a renewed interest in this property stems from a recent result of Goli\'nska and Wegner \cite{GoWe16} stating that $m$-topologizable operators on sequentially complete locally convex spaces generate uniformly continuous semigroups of operators. It should be noted that in contrast to Banach spaces, on arbitrary locally convex spaces in general not every continuous linear operator generates a strongly continuous semigroup, see \cite{FrJoKaWe14}.

Our results for power boundedness for weighted composition operators $C_{\psi,\varphi}$ are sharp when $\psi,\varphi$ are univariate polynomials. They allow to provide natural examples of infinite dimensional subspaces of $\mathcal{L}(\mathscr{S}(\R))$ consisting entirely of power bounded operators, or -- except the zero operator -- entirely of non-power bounded operators, respectively. Concrete examples of such infinite dimensional vector spaces are $\{C_{\psi,\varphi}: \varphi(x)=x^2+1, \psi \text { polynomial}\}$ (Theorem \ref{main}) and $\{C_{\psi,\varphi}: \varphi(x)=(1/2)x, \psi \text { polynomial}\}$ (Proposition \ref{prop:powerbounded for ax+b}), respectively. We point out that examples of the second kind cannot occur for operators defined on Banach spaces.

The article is organized as follows. In section \ref{sec:weighted compositions on S} we characterize those pairs $(\psi,\varphi)$ for which the corresponding weighted composition operator $C_{\psi,\varphi}$ acts on $\mathscr{S}(\R^d)$. Under mild additional assumptions on $\psi$ and $\varphi$, in section \ref{sec:small decay multiplies}, we give a characterization for the latter property which is easier to check in many situations including composition operators. In section \ref{sec: power boundedness} we characterize power boundedness and ($m$-)topologizability of $C_{\psi,\varphi}$ on $\mathscr{S}(\R^d)$. As a concrete example we show power boundedness of $C_{\exp,\exp}$ on $\mathscr{S}(\R)$, i.e.\ $\psi=\varphi=\exp$. It should be noted that neither the composition by $\exp$ nor the multiplication by $\exp$ acts on $\mathscr{S}(\R)$. As an application of our findings in section \ref{sec: power boundedness}, in section \ref{sec:polynomials} we study power boundedness of weighted composition operators on $\mathscr{S}(\R)$ for the case that $\varphi$ is a (univariate) polynomial. Among others, we prove that power boundedness of translation operators, i.e. $\varphi(x) = x+b$, $b \neq 0$, can be achieved by multiplication with constants of modulus strictly smaller than 1. We also show that for $\text{deg}(\varphi)\geq 2$, power boundedness of the composition operator $C_\varphi$ is equivalent to the power boundedness and/or uniform mean ergodicity of the weighted composition operators $C_{\psi,\varphi}$ for arbitrary $\psi\in\mathscr{O}_M(\R)$. In the short final section \ref{sec:supercyclicity}, we apply arguments from section \ref{sec:polynomials} to show that a univariate polynomial $\varphi$ is necessarily a translation whenever there is $\psi\in \mathscr{O}_M(\R)$ such that $C_{\psi,\varphi}$ is weakly supercyclic, i.e.\ there is $f\in \mathscr{S}(\R)$ with $\{\lambda C_{\psi,\varphi}^nf;\,\lambda\in\C, n\in\N\}$ is weakly dense in $\mathscr{S}(\R)$. This complements recent results on hypercyclicity of weighted translation operators on $\mathscr{S}(\R)$ by Goli\'nski and Przestacki \cite{GoPr20}.

Throughout, we use standard notation from functional analysis \cite{MeVo97} and dynamics of linear operators on locally convex spaces \cite{BM, BoPaRi11}.

\section{Weighted composition operators on $\mathscr{S}(\R^d)$}\label{sec:weighted compositions on S}

The main purpose of this section is to characterize those pairs $(\psi,\varphi)$, $\psi\in C^\infty(\R^d)$ and $\varphi:\R^d\rightarrow\R^d$ smooth mapping, for which the corresponding weighted composition operator $C_{\psi,\varphi}$ acts on $\mathscr{S}(\R^d)$, where we say that $C_{\psi,\varphi}(f)=\psi\cdot(f\circ\varphi)$ \textit{acts on} $\mathscr{S}(\R^d)$ if $C_{\psi,\varphi}f\in\mathscr{S}(\R^d)$ for every $f\in \mathscr{S}(\R^d)$. Obviously, in this case $C_{\psi,\varphi}$ is a linear mapping on $\mathscr{S}(\R^d)$. As usual, we say that $\varphi$ is a \textit{symbol for} $\mathscr{S}(\R^d)$ if $C_\varphi:=C_{1,\varphi}$ acts on $\mathscr{S}(\R^d)$. Additionally, we recall that $\psi$ is a \textit{multiplier for} $\mathscr{S}(\R^d)$ if $M_\psi:=C_{\psi,\operatorname{id}_{\R^d}}$ acts on $\mathscr{S}(\R^d)$. We shall see in Example \ref{ex: exp} below that neither $C_{\varphi}$ acting on $\mathscr{S}(\mathbb{R})$ nor $\psi$ being a multiplier for $\mathscr{S}(\mathbb{R})$ is necessary for $C_{\psi,\varphi}$ to act on $\mathscr{S}(\mathbb{R})$.

We first fix some notation which is valid throughout this paper. As usual, for a smooth function $f:\R^d\rightarrow\C$ and a multi index $\alpha=(\alpha_1,\ldots,\alpha_d)\in \N_0^d$ we write $f^{(\alpha)}(x)=\partial^\alpha f(x)=\frac{\partial^{|\alpha|}}{\partial x_1^{\alpha_1}\cdots\partial x_d^{\alpha_d}}f(x)$, where as usual $|\alpha|=\sum_{j=1}^d\alpha_j$. While we also denote the Euclidean norm of $x\in\R^d$ by $|x|$, it will always be clear from the context whether we refer to the length of a multi index or the Euclidean norm of a vector. Next, we recall some notation from \cite{CoSa1996} used in the multivariate version of the Fa\`a di Bruno formula. On the set  $\N_0^d$ of multi indices, for $\alpha=(\alpha_1,\ldots,\alpha_d), \beta=(\beta_1,\ldots,\beta_d)$, we write $\alpha\prec\beta$ provided one of the following holds:
\begin{itemize}
	\item[(i)] $|\alpha|<|\beta|$,
	\item[(ii)] $|\alpha|=|\beta|$ and $\alpha_1<\beta_1$, or
	\item[(iii)] $|\alpha|=|\beta|$, $\alpha_1=\beta_1,\ldots,\alpha_k=\beta_k$ and $\alpha_{k+1}<\beta_{k+1}$ for some $1\leq k<d$.
\end{itemize}
Moreover, for $\beta\in\N_0^d\backslash\{0\}$ and $\lambda\in\N_0^d$ we define the set
\begin{eqnarray*}
	p(\beta,\lambda)&=&\left\{(k_1,\ldots,k_{|\beta|};\ell_1,\ldots,\ell_{|\beta|})\in\N_0^{2|\beta|d}:\,\text{for some }1\leq s\leq |\beta|,\textcolor{white}{\sum_{j=1}^{|\beta|}}\right.\\
	&&\quad k_j=\ell_j=0\text{ for }1\leq j\leq |\beta|-s; |k_j|>0\text{ for }|\beta|-s+1\leq j\leq |\beta|,\text{ and}\\
	&&\quad \left. 0\prec\ell_{|\beta|-s+1}\prec\cdots\prec\ell_{|\beta|}\text{ are such that }\sum_{j=1}^{|\beta|}k_j=\lambda, \sum_{j=1}^{|\beta|}|k_j|\ell_j=\beta\right\}.
\end{eqnarray*}
It is straightforward to show $|\lambda|\leq |\beta|$ whenever $p(\beta,\lambda)\neq\emptyset$. Then, for a smooth function $f:\R^d\rightarrow\C$, a smooth mapping $\varphi:\R^d\rightarrow\R^d$, and for every $\beta\in\N_0^d\backslash\{0\}$ we have
$$\left(f\circ \varphi\right)^{(\beta)}(x)=\sum_{\substack{\lambda\in\N_0^d\\1\leq |\lambda|\leq |\beta|}}f^{(\lambda)}(\varphi(x))\sum_{p(\beta,\lambda)}\beta!\prod_{j=1}^{|\beta|}\frac{\left(\varphi^{(\ell_j)}(x)\right)^{k_j}}{k_j!\left(\ell_j!\right)^{|k_j|}}$$
(see \cite[Remark 2.2]{CoSa1996}), where $\varphi^{(\ell_j)}(x)=(\varphi_1^{(\ell_{j})}(x),\ldots,\varphi_d^{(\ell_j)}(x))$ and where for $y=(y_1,\dots,y_d)\in\C^d$ and $k=(k_1,\ldots,k_d)\in\N_0^d$, as usual $y^k=y_1^{k_1}\cdots y_d^{k_d}$, so that
$$\left(\varphi^{(\ell_j)}(x)\right)^{k_j}=\prod_{i=1}^{d}\left(\varphi_i^{(\ell_{j})}(x)\right)^{k_{j,i}}=\prod_{i=1}^d\left(\frac{\partial^{|\ell_j|}}{\partial x_1^{\ell_{j,1}}\cdots\partial x_d^{\ell_{j,d}}}\varphi_i(x)\right)^{k_{j,i}},$$
with $\ell_j=(\ell_{j,1},\ldots,\ell_{j,d}), k_j=(k_{j,1},\ldots,k_{j,d})\in\N_0^d$. It follows immediately from the definition of the set $p(\beta,\lambda)$ that $p(\beta,0)=\emptyset$. Thus, employing the usual convention that the sum of summands indexed by the empty set equals zero, for every $\beta\in\N_0^d\backslash\{0\}$ we have
$$\left(f\circ \varphi\right)^{(\beta)}(x)=\sum_{\substack{\lambda\in\N_0^d\\0\leq |\lambda|\leq |\beta|}}f^{(\lambda)}(\varphi(x))\sum_{p(\beta,\lambda)}\beta!\prod_{j=1}^{|\beta|}\frac{\left(\varphi^{(\ell_j)}(x)\right)^{k_j}}{k_j!\left(\ell_j!\right)^{|k_j|}}.$$
Abusing notation, for $\beta=0$, we further set $p(0,\lambda)=\emptyset$ whenever $\lambda\in\N_0^d\backslash \{0\}$ and $p(0,0)=\{(0,0)\}$ ($0\in\N_0^d$) so that, by the usual convention $0!=1$ and $0^0=1$
$$\sum_{p(0,\lambda)}0!\prod_{j=1}^{|0|}\frac{\left(\varphi^{(\ell_j)}(x)\right)^{k_j}}{k_j!\left(\ell_j!\right)^{|k_j|}}=\begin{cases}
	0,& \lambda\neq 0,\\
	1,& \lambda=0,
\end{cases}$$
where we also applied the usual convention that the product of factors indexed by the empty set equals one. Thus, for $f$ and $\varphi$ as above, for $\psi\in C^\infty(\R^d)$, and for every multi index $\alpha\in\N_0^d$, by applying Leibniz' rule and Fa\`{a} di Bruno's formula, after reordering, we get
\begin{eqnarray*}
	(\psi\cdot (f\circ \varphi))^{(\alpha)}(x)&=&\sum_{\substack{\beta\in\N_0^d\\ \beta\leq \alpha}}\sum_{\substack{\lambda\in\N_0^d\\ 0\leq |\lambda|\leq|\beta|}} f^{(\lambda)}(\varphi(x)){\alpha \choose\beta} \psi^{(\alpha-\beta)}(x)\sum_{p(\beta,\lambda)}\beta!\prod_{j=1}^{|\beta|}\frac{\left(\varphi^{(\ell_j)}(x)\right)^{k_j}}{k_j!\left(\ell_j!\right)^{|k_j|}}\nonumber\\
	&=&\sum_{\substack{\beta\in\N_0^d\\ |\beta|\leq |\alpha|}}\sum_{\substack{\lambda\in\N_0^d\\ 0\leq |\lambda|\leq|\beta|}} f^{(\lambda)}(\varphi(x)){\alpha \choose\beta} \psi^{(\alpha-\beta)}(x)\sum_{p(\beta,\lambda)}\beta!\prod_{j=1}^{|\beta|}\frac{\left(\varphi^{(\ell_j)}(x)\right)^{k_j}}{k_j!\left(\ell_j!\right)^{|k_j|}}\\
	&=&\sum_{\substack{\lambda\in\N_0^d\\ 0\leq |\lambda|\leq|\alpha|}} f^{(\lambda)}(\varphi(x))\sum_{\substack{\beta\in\N_0^d\\ \beta\leq \alpha, |\lambda|\leq |\beta|}}{\alpha \choose\beta} \psi^{(\alpha-\beta)}(x)\sum_{p(\beta,\lambda)}\beta!\prod_{j=1}^{|\beta|}\frac{\left(\varphi^{(\ell_j)}(x)\right)^{k_j}}{k_j!\left(\ell_j!\right)^{|k_j|}},\nonumber
\end{eqnarray*}
where we have used that ${\alpha\choose\beta}=\prod_{j=1}^d{\alpha_j\choose\beta_j}=0$ if there are $\alpha_j,\beta_j\in\N_0$ with $\beta_j>\alpha_j$ since ${z\choose m}=\frac{z(z-1)\cdots(z-m+1)}{m!}$ for $z\in\C$ and $m\in\N_0$.



For $\psi\in C^\infty(\R^d)$, a smooth mapping $\varphi:\R^d\rightarrow\R^d$ and $\alpha,\lambda\in\N_0^d$ with $|\lambda|\leq |\alpha|$ we denote
\begin{equation}\label{eq:definition F}
	\Fmult(x):=\sum_{\substack{\beta\in\N_0^d\\ \beta\leq \alpha, |\lambda| \leq |\beta|}}{\alpha \choose\beta} \psi^{(\alpha-\beta)}(x)\sum_{p(\beta,\lambda)}\beta!\prod_{j=1}^{|\beta|}\frac{\left(\varphi^{(\ell_j)}(x)\right)^{k_j}}{k_j!\left(\ell_j!\right)^{|k_j|}}
\end{equation}
and thus
\begin{equation} 
	\label{eq: Leibniz and Faa di Bruno}
	(\psi \cdot (f \circ \varphi))^{(\alpha)}(x) = \sum_{\substack{\lambda\in\N_0^d\\ 0\leq |\lambda|\leq|\alpha|}} f^{(\lambda)}(\varphi(x)) \Fmult(x).
\end{equation}

Now we are ready to prove a first technical lemma which will also be used in section \ref{sec: power boundedness}.

\begin{lemma}
	\label{techlemma}
	Let $I$ be a non-empty set, $(\psi_i)_{i\in I}\in C^\infty(\R^d)^I$ and let $\varphi_i:\R^d\rightarrow\R^d$ be smooth mappings, $i\in I$. Assume that there are $\alpha,\lambda \in\N_0^d$ with $|\lambda|\leq |\alpha|$ and $p>0$ such that
	\begin{equation}\label{eq:tech lemma}
		\forall\,q>0:\,\sup_{i\in I}\sup_{x\in\R^d} \frac{(1+|x|)^p}{(1+|\varphi_i(x)|)^q} |\Fmulti(x)|=\infty.
	\end{equation}
	Then there is $f\in \mathscr{S}(\R^d)$ satisfying 
	$$\sup_{i\in I}\sup_{x\in\R^d} (1+|x|)^p |(\psi_i\cdot (f\circ\varphi_i))^{(\alpha)}(x) |=\infty.$$
\end{lemma}

\begin{proof}
	Let $\lambda_0\in\N_0^d$ be the minimum with respect to the linear ordering $\prec$ of $\N_0^d$ of the finite set of $\lambda\in\N_0^d$ with $|\lambda|\leq |\alpha|$ for which \eqref{eq:tech lemma} holds. Let $(x_j)_{j\in\N}$, $(i_j)_{j\in\N}$ be sequences in $\R^d$ and $I$, respectively, such that $|x_{j+1}|>|x_{j}|+1$ and
	\begin{equation}
		\label{unbounded2}
		\frac{(1+|x_j|)^p}{(1+|\varphi_{i_j}(x_j)|)^j} |\Fmultoj(x_j)| >j.
	\end{equation}
	At this step, we continue the proof by considering two cases. First we assume $(|\varphi_{i_j}(x_j)|)_{j\in\N}$ to be unbounded. By making an abuse of notation and identifying $(x_j)_{j\in\N}$ with a subsequence, we can assume  $|x_{j+1}|>|x_{j}|+1$ and $|\varphi_{i_{j+1}}(x_{j+1})|>|\varphi_{i_j}(x_{j})|+1$ for every $j\in\N$, and let $(l(j))_{j\in\N}$ be a strictly increasing sequence of natural numbers such that 
	\begin{equation}
		\label{unbounded}
		\lim_{j\rightarrow\infty}  \frac{(1+|x_j|)^p}{(1+|\varphi_{i_j} (x_j)|)^{l(j)}} |\Fmultoj(x_j)| =\infty.
	\end{equation}
	
	Fix $\varrho\in \mathscr{D}(B(0,1/2))$, $\varrho^{(\lambda_0)}(0)=1$ and $\varrho^{(\lambda)}(0)=0$ for $\lambda\in\N_0^d\backslash\{\lambda_0\}$. We define
	$$f(x):=\sum_{j\in\N}\frac{\varrho(x-\varphi_{i_j}(x_j))}{(1+|\varphi_{i_j}(x_j)|)^{l(j)}}.$$
	Since the summands which define $f$ are smooth functions with mutually disjoint compact supports and because $\lim_j l(j)=\infty$, it is standard to check $f\in \mathscr{S}(\R^d)$. Moreover, by \eqref{eq: Leibniz and Faa di Bruno} and the definition of $\Fmult$	  
	\begin{eqnarray*}
		(1+|x_j|)^p|(\psi_{i_j}\cdot (f\circ \varphi_{i_j}))^{(\alpha)}(x_j)|&=&(1+|x_j|)^p\left|\sum_{\substack{\lambda\in\N_0^d\\ 0\leq|\lambda|\leq|\alpha|}} f^{(\lambda)}(\varphi_{i_j}(x_j))\Fmultj(x_j)\right|\\ &=&(1+|x_j|)^p|f^{(\lambda_0)}(\varphi_{i_j}(x_j))\Fmultoj(x_j)|\\
		&=&\frac{(1+|x_j|)^p}{(1+|\varphi_{i_j}(x_j)|)^{
				l(j)}} |\Fmultoj(x_j)| 
	\end{eqnarray*}
	so that by \eqref{unbounded}
	$$\sup_{i\in I}\sup_{x\in\R^d}(1+|x|)^p |(\psi_i\cdot (f\circ\varphi_i))^{(\alpha)}(x)|\geq\sup_{j\in\N}(1+|x_j|)^p|(\psi_{i_j}\cdot (f\circ \varphi_{i_j}))^{(\alpha)}(x_j)|=\infty$$
	as desired.
	
	Next, let us suppose that $(|\varphi_{i_j}(x_j)|)_{j\in\N}$ is bounded. 
	We observe that the choice of $\lambda_0$ implies the existence of $q>0$ and $C\in (0,\infty)$ such that, for each $\lambda\in\N_0^d$ with $\lambda\prec\lambda_0$,
	$$
	\sup_{x\in\R^d}\frac{(1+|x|)^p}{(1+|\varphi_{i_j}(x)|)^q} |\Fmultj(x)|\leq C.
	$$
	Next, we fix $f\in \mathscr{D}(\R^d)$ such that $f(x)=x^{\lambda_0}/\lambda_0!$ in a neighborhood of the bounded subset $\{\varphi_{i_j}(x_j);\,j\in\N\}$ of $\R^d$. Then, a moment's reflection reveals that $f^{(\lambda)}\equiv 0$ in a neighborhood of $\{\varphi_{i_j}(x_j);\,j\in\N\}$ whenever $\lambda_0\prec\lambda$ and obviously, for every $q>0$ there is $M\in (0,\infty)$ with
	$$\sup_{\substack{j\in\N,\\ \lambda\in\N_0^d, |\lambda|\leq |\alpha|}}\left|f^{(\lambda)}(\varphi_{i_j}(x_j))\right|(1+|\varphi_{i_j}(x_j)|)^q\leq M.$$
	
	Now we have by \eqref{eq: Leibniz and Faa di Bruno} and \eqref{unbounded2}, for $j\in\N$,
	\begin{eqnarray*}
		&&(1+|x_j|)^p|(\psi_{i_j}\cdot (f\circ \varphi_{i_j}))^{(\alpha)}(x_j)|=(1+|x_j|)^p\left|\sum_{\substack{\lambda\in\N_0^d\\ 0\leq|\lambda|\leq|\alpha|}} f^{(\lambda)}(\varphi_{i_j}(x_j))\Fmultj(x_j)\right|\\
		&\geq& (1+|x_j|)^p\left|f^{(\lambda_0)}(\varphi_{i_j}(x_j))\Fmultoj(x_j)\right|\\
		&&-\sum_{\lambda\prec\lambda_0}\left|f^{(\lambda)}(\varphi_{i_j}(x_j))\right|(1+|\varphi_{i_j}(x_j)|)^q\frac{(1+|x_j|)^p}{(1+|\varphi_{i_j}(x_j)|)^q}\left|\Fmultj(x_j)\right|\\
		&\geq&(1+|\varphi_{i_j}(x_j)|)^j\frac{(1+|x_j|)^p}{(1+|\varphi_{i_j}(x_j)|)^j}\left|\Fmultoj(x_j)\right|-\sum_{\lambda\prec\lambda_0}MC \geq j-\sum_{\lambda\prec\lambda_0}MC 
	\end{eqnarray*}
	so that again
	$$\sup_{i\in I}\sup_{x\in\R^d}(1+|x|)^p |(\psi_i\cdot (f\circ\varphi_i))^{(\alpha)}(x)|\geq\sup_{j\in\N}(1+|x_j|)^p|(\psi_{i_j}\cdot (f\circ \varphi_{i_j}))^{(\alpha)}(x_j)|=\infty$$
	as desired.
\end{proof}

\begin{theorem}\label{theo: weighted compositions characterized}
	Let $\psi\in C^\infty(\R^d)$ and let $\varphi:\R^d\rightarrow\R^d$ be smooth. Then, the following are equivalent.
	\begin{itemize}
		\item[(i)] $\psi\cdot (f\circ\varphi)\in \mathscr{S}(\R^d)$ for every $f\in \mathscr{S}(\R^d)$.
		\item[(ii)] The weighted composition operator $C_{\psi,\varphi}:\mathscr{S}(\R^d)\rightarrow \mathscr{S}(\R^d), f\mapsto \psi\cdot (f\circ\varphi)$ is correctly defined and continuous.
		\item[(iii)] For all $\alpha,\lambda\in\N_0^d$ with $|\lambda|\leq|\alpha|$ and for every $p>0$ there exists $q>0$ such that
		$$\sup_{x\in\R^d} \frac{(1+|x|)^p}{(1+|\varphi(x)|)^q} |\Fmult(x)|<\infty.$$
	\end{itemize}
\end{theorem}

\begin{proof}
	Obviously, (ii) implies (i), and (i) implies (iii) by Lemma~\ref{techlemma} applied to $I=\{1\}$ and $\psi_1=\psi, \varphi_1=\varphi$. Additionally, using \eqref{eq: Leibniz and Faa di Bruno} it is straightforward to show that (iii) implies (ii).
\end{proof}

The evaluation of condition (iii) from the above theorem for concrete $\psi,\varphi$ might be quite involved  due to the rather complicated expression for $\Fmult$ in \eqref{eq:definition F}. Therefore, we now give necessary and sufficient conditions for $C_{\psi,\varphi}$ to act on $\mathscr{S}(\R^d)$ which are easier to evaluate. Since for $\alpha\in \N_0^d$ and $\lambda=0$ it holds
$\Fmult(x)=\psi^{(\alpha)}(x)$, Theorem~\ref{theo: weighted compositions characterized} (iii) immediately implies the following necessary condition.

\begin{corollary}
	\label{alpha0cor}
	If the weighted composition operator $C_{\psi,\varphi}:\mathscr{S}(\R^d)\rightarrow \mathscr{S}(\R^d), f\mapsto \psi\cdot (f\circ\varphi)$ is well defined then for each $\alpha\in\N_0^d$ and $p>0$ there is $q>0$ such that 
	$$\sup_{x\in\R^d}\frac{(1+|x|)^p}{(1+|\varphi(x)|)^q}|\psi^{(\alpha)}(x)|<\infty.$$
\end{corollary}

Theorem \ref{theo: weighted compositions characterized} immediately implies the next sufficient condition for $C_{\psi,\varphi}$ to act on $\mathscr{S}(\R^d)$. It will be shown in Theorem~\ref{theo: small decay} below that under some mild additional assumptions on $\psi,\varphi$ this sufficient condition is also necessary.

\begin{corollary}
	\label{sufficient}
	Let $\psi\in C^\infty(\R^d)$ and let $\varphi:\R^d\rightarrow\R^d$ be smooth. Assume that for each $\alpha\in\N_0^d$ and $p>0$ there is $q>0$ such that 
	$$\sup_{x\in\R^d}\frac{(1+|x|)^p}{(1+|\varphi(x)|)^q}|\psi^{(\alpha)}(x)|<\infty,$$
	and for each $\alpha\in\N_0^d$ there is $q>0$ such that 
	$$\sup_{x\in\R^d}\frac{1}{(1+|\varphi(x)|)^q}|\varphi^{(\alpha)}(x)|<\infty.$$
	Then $C_{\psi,\varphi}$ acts on $\mathscr{S}(\R^d)$.
\end{corollary}

\begin{remark}\label{rem: multiplication}
	For the special case of a multiplication operator, i.e.\ $\varphi(x)=x$ for all $x\in\R^d$, for $k_j,\ell_j\in\N_0^d$ we have
	$$\left(\varphi^{(\ell_j)}(x)\right)^{k_j}=\prod_{i=1}^d\left(\partial^{\ell_j} x_i\right)^{k_j}=\begin{cases}
		1,& \text{ if }\ell_j=k_j=0,\\
		0,& \text{ if }|\ell_j|\geq 2,\\
		0,& \text{ if }|\ell_j|=1, k_j\notin\text{span}\{\ell_j\},\\
		1,& \text{ if }|\ell_j|=1, k_j\in\text{span}\{\ell_j\}.
	\end{cases}$$
	Therefore, defining
	$$p_0(\beta,\lambda):=\{(k_1,\ldots,k_{|\beta|};\ell_1,\ldots,\ell_{|\beta|})\in p(\beta,\lambda):\,|\ell_{|\beta|}|\leq 1\text{ and }k_j\in\text{span}\{\ell_j\}\text{ for all }1\leq j\leq |\beta|\}$$
	it follows that
	$$\Fmult(x)=\sum_{\substack{\beta\in\N_0^d\\\beta\leq \alpha, |\lambda|\leq|\beta|}} {\alpha\choose\beta}\psi^{(\alpha-\beta)}(x)\sum_{p_0(\beta,\lambda)}\beta!\prod_{j=1}^{|\beta|}\frac{1}{k_j!}.$$
	In order to continue, let $e_t=(\delta_{j,t})_{1\leq j\leq d}$ (Kronecker's $\delta$), $1\leq t\leq d$, be the standard basis vectors of $\R^d$ and for $\beta\in\N_0^d\backslash\{0\}$, $\beta=\sum_{t=1}^d \beta_t e_t$ let $I(\beta):=\{1\leq t\leq d:\,\beta_t\neq 0\}$ so that $\emptyset\neq I(\beta)=\{t_1,\ldots,t_\beta\}$ with $1\leq t_1<\ldots<t_\beta\leq d$. Moreover, we denote the number of elements of $I(\beta)$ by $|I(\beta)|$. With this notation, for $\beta\in\N_0^d\backslash\{0\}$ we conclude
	\begin{eqnarray*}
		p_0(\beta,\lambda)&=&\left\{(k_1,\ldots,k_{|\beta|};\ell_1,\ldots,\ell_{|\beta|})\in p(\beta,\lambda):\,|\ell_{|\beta|}|= 1, k_j\in\text{span}\{\ell_j\}\text{ for all }1\leq j\leq |\beta|,\textcolor{white}{\sum_{j=1}^{|\beta|}|k_j|\ell_j=\beta}\right.\\
		&&\quad k_j=\ell_j=0\text{ for }1\leq j\leq |\beta|-|I(\beta)|, \ell_{|\beta|-|I(\beta)|+1}=e_{t_1},\ldots,\ell_{|\beta|}=e_{t_\beta},\\
		&&\quad\left. |k_j|>0 \text{ for }|\beta|-|I(\beta)|+1\leq j\leq |\beta|, \sum_{j=1}^{|\beta|}k_j=\lambda, \sum_{j=1}^{|\beta|}|k_j|\ell_j=\beta\right\}\\
		&=&\begin{cases}
			\quad\emptyset,&\text{ if }\beta\neq\lambda,\\
			\{(k_1,\ldots,k_{|\beta|};\ell_1,\ldots,\ell_{|\beta|})\in\N_0^{2|\beta|d}:\,k_j=\ell_j=0, 1\leq j\leq |\beta|-|I(\beta)|,&\\ \quad \ell_{|\beta|-|I(\beta)|+1}=e_{t_\beta},\ell_{|\beta|-|I(\beta)|+2}=e_{t_\beta-1},\ldots,\ell_{|\beta|}=e_{t_1},&\\ \quad k_{|\beta|-|I(\beta)|+1}=\beta_{t_\beta}e_{t_\beta},k_{|\beta|-|I(\beta)|+2}=\beta_{t_\beta-1}e_{t_\beta-1},\ldots,k_{|\beta|}=\beta_{t_1}e_{t_1}\},&\text{ if }\beta=\lambda,
		\end{cases}
	\end{eqnarray*}
	so that for $\lambda\neq 0$
	$$\Fmult(x)={\alpha\choose\lambda}\psi^{(\alpha-\lambda)}(x)$$
	which also holds true for $\lambda= 0$ (recall that the above multinomial coefficient is zero whenever $\lambda\nleq\alpha$). 
	
	Hence, Theorem~\ref{theo: weighted compositions characterized} characterizes those $\psi\in C^\infty(\R^d)$ for which the corresponding multiplication operator $M_\psi:\mathscr{S}(\R^d)\rightarrow \mathscr{S}(\R^d), f\mapsto \psi f$ is correctly defined by the property that for every $\alpha,\lambda\in\N_0^d$ with $\lambda\leq \alpha$ and every $p>0$ there is $q>0$ such that
	$$\sup_{x\in\R^d}\frac{(1+|x|)^p}{(1+|x|)^q}|\psi^{(\alpha-\lambda)}(x)|<\infty.$$
	Obviously, this holds precisely when for each $\gamma\in\N_0^d$ there is $r>0$ such that
	$$\sup_{x\in\R^d}(1+|x|)^{-r}|\psi^{(\gamma)}(x)|<\infty.$$
	Thus, for the special case $\varphi(x)=x$, Theorem~\ref{theo: weighted compositions characterized} gives the well known characterization of the space of multipliers of $\mathscr{S}(\R^d)$ as $\mathscr{O}_M(\R^d)$. 
	
	Additionally, by Corollary \ref{alpha0cor}, if $\varphi$ is a non-constant elliptic polynomial, $C_{\psi,\varphi}$ acts on $\mathscr{S}(\R^d)$ if and only if $\psi\in\mathscr{O}_M(\R^d)$. Therefore, if $d=1$, this equivalence holds whenever $\varphi$ is a non-constant univariate polynomial.
\end{remark}

\begin{remark}\label{rem: one dimensional}
	Next, we consider the special case $d=1$ in Theorem~\ref{theo: weighted compositions characterized}. We define for $\beta,\lambda\in\N_0\backslash\{0\}$ with $\lambda\leq\beta$
	$$q(\beta,\lambda)=\left\{(i_1,\ldots,i_\beta)\in\N_0^{\beta}:\,\sum_{j=1}^\beta i_j=\lambda, \sum_{j=1}^\beta j i_j=\beta\right\},$$
	and for $(i_1,\ldots,i_\beta)\in q(\beta,\lambda)$ we set
	$$L(i_1,\ldots,i_\beta):=\{ 1 \leq j \leq \beta :\,i_j\neq 0\}\text{ and }s(i_1,\ldots,i_\beta):=|L(i_1,\ldots,i_\beta)|.$$
	Then, $\emptyset\neq L(i_1,\ldots,i_\beta)\subset\{1,\ldots,\beta\}$ and $1\leq s(i_1,\ldots,i_\beta)\leq \beta$. It is straightforward that the correspondence $q(\beta,\lambda) \rightarrow p(\beta,\lambda)$ which maps $(i_1, \ldots,i_{\beta})$ into
	$$(0,\ldots,0,i_{\min L(i_1,\ldots,i_\beta)},\ldots,i_{\max L(i_1,\ldots,i_\beta)};0,\ldots,0,\min L(i_1,\ldots,i_\beta),\ldots,\max L(i_1,\ldots,i_\beta))$$
	is correctly defined and bijective, where we have (twice) $\beta-s$ zeros. With this it follows for $d=1$ that for $\psi,\varphi\in C^\infty(\R)$, $\varphi$ real valued, $\alpha,\lambda\in\N_0$ with $\lambda\leq \alpha$
	$$\Fmult(x)=\sum_{\beta=\lambda}^\alpha {\alpha\choose\beta} \psi^{(\alpha-\beta)}(x)\sum_{\substack{i_1,\ldots,i_\beta\in\N_0,\\ i_1+\cdots+i_\beta=\lambda,\\ i_1+2i_2+\cdots+\beta i_\beta=\beta}} \beta! \prod_{j=1}^\beta \frac1{i_j!} \left(\frac{\varphi^{(j)}(x)}{j!}\right)^{i_j}.$$
	Noticing that $i_j=0$ for $j>\beta-\lambda+1$ the above simplifies to
	\begin{eqnarray*}
		\Fmult(x)&=&\sum_{\beta=\lambda}^\alpha {\alpha\choose\beta} \psi^{(\alpha-\beta)}(x)\sum_{\substack{i_1,\ldots,i_{\beta-\lambda+1}\in\N_0,\\ i_1+\cdots+i_{\beta-\lambda+1}=\lambda,\\ i_1+2i_2+\cdots+\beta i_{\beta-\lambda+1}=\beta}} \beta! \prod_{j=1}^{\beta-\lambda+1} \frac1{i_j!} \left(\frac{\varphi^{(j)}(x)}{j!}\right)^{i_j}\\
		&=&\sum_{\beta=\lambda}^\alpha {\alpha\choose\beta} \psi^{(\alpha-\beta)}(x) B_{\beta,\lambda}\left(\varphi'(x),\ldots,\varphi^{(\beta-\lambda+1)}(x)\right),
	\end{eqnarray*}
	where $B_{\beta,\lambda}$ denotes the corresponding Bell polynomial, i.e.\
	\begin{equation}\label{eq: Bell Polynomial}
		\forall\,x_1,\ldots,x_{\beta-\lambda+1}\in\R:\,B_{\beta,\lambda}(x_1,\ldots,x_{\beta-\lambda+1})=\sum_{\substack{i_1,i_2\ldots, i_{\beta-\lambda+1}\in\N_0,\\ i_1+i_2+\cdots+i_{\beta-\lambda+1}=\lambda,\\i_1+2i_2+\cdots +(\beta-\lambda+1)i_{\beta-\lambda+1}=\beta}} \beta! \prod_{r=1}^{\beta-\lambda+1} \frac1{i_r!} \left(\frac{x_r}{r!}\right)^{i_r}
	\end{equation}
	for $\beta,\lambda\in\N$ with $\lambda\leq \beta$, and $B_{0,0}=1$, $B_{\beta,0}=0$, $\beta\in\N$.
\end{remark}

\begin{corollary}\label{cor: one dimensional weighted composition characterized}
	Let $\psi,\varphi\in C^\infty(\R)$, $\varphi$ be real valued. Then, the following are equivalent.
	\begin{itemize}
		\item[(i)] $\psi\cdot(f\circ\varphi)\in \mathscr{S}(\R)$ for every $f\in \mathscr{S}(\R)$.
		\item[(ii)] The weighted composition operator $C_{\psi,\varphi}:\mathscr{S}(\R)\rightarrow \mathscr{S}(\R), f\mapsto \psi\cdot(f\circ\varphi)$ is correctly defined and continuous.
		\item[(iii)] For every $\alpha,\beta,\lambda\in\N_0$ with $\alpha\geq \beta\geq\lambda$ and for each $p>0$ there is $q>0$ such that
		$$\sup_{x\in\R}\frac{(1+|x|)^p}{(1+|\varphi(x)|)^q}\left|\psi^{(\alpha-\beta)}(x) B_{\beta,\lambda}\left(\varphi'(x),\ldots,\varphi^{(\beta-\lambda+1)}(x)\right)\right|<\infty.$$
	\end{itemize}
\end{corollary}

\begin{proof}
	By Theorem~\ref{theo: weighted compositions characterized} and Remark~\ref{rem: one dimensional}, (i) and (ii) are equivalent and hold precisely when for every $\alpha,\beta,\lambda\in \N_0$ with $\alpha\geq\beta\geq\lambda$ and each $p>0$ there is $q>0$ with
	$$\sup_{x\in\R}\frac{(1+|x|)^p}{(1+|\varphi(x)|)^q}\left|\sum_{\beta=\lambda}^\alpha {\alpha\choose\beta}\psi^{(\alpha-\beta)}(x) B_{\beta,\lambda}\left(\varphi'(x),\ldots,\varphi^{(\beta-\lambda+1)}(x)\right)\right|<\infty.$$
	Evaluating the latter condition first for $\lambda=\alpha$, then for $\lambda=\alpha-1$ etc.\ finally shows that (i), (ii), and (iii) are equivalent.
\end{proof}

\begin{remark}\label{rem: one dimensional composition operators}
	For a real valued $\varphi\in C^\infty(\R)$ it follows from Corollary~\ref{cor: one dimensional weighted composition characterized} with $\psi\equiv 1$ that the corresponding composition operator $C_\varphi f= f\circ\varphi$ acts on $\mathscr{S}(\R)$ if and only if 
	\begin{equation}\label{eq: one dimensional composition operators}
		\forall\,\alpha,\lambda\in\N_0, \alpha\geq \lambda\,\forall\,p>0\,\exists\,q>0:\,\sup_{x\in\R}\frac{(1+|x|)^p}{(1+|\varphi(x)|)^q}\left|B_{\alpha,\lambda}\left(\varphi'(x),\ldots,\varphi^{(\alpha-\lambda+1)}(x)\right)\right|<\infty.
	\end{equation}
	Since $B_{0,0}\equiv 1$ and $B_{\alpha,1}(x_1,\ldots,x_\alpha)=x_\alpha$, condition \eqref{eq: one dimensional composition operators} implies
	\begin{itemize}
		\item there is $k>0$ such that $|\varphi(x)|\geq |x|^{1/k}$ whenever $|x|\geq k$, and
		\item for every $\alpha\in \N_0$ there are $C,q>0$ such that $|\varphi^{(\alpha)}(x)|\leq C (1+|\varphi(x)|)^q$.
	\end{itemize}
	On the other hand, the previous two conditions easily imply condition \eqref{eq: one dimensional composition operators} so that Corollary~\ref{cor: one dimensional weighted composition characterized} gives an alternate characterization to that in \cite[Theorem 2.3]{GaJo18}. Furthermore, the characterization in \cite[Theorem 2.3]{GaJo18} is also valid for several variables (see~\cite[Remark 2.4(1)]{GaJo18}), which can be deduced from Theorem~\ref{theo: weighted compositions characterized}, too, in a similar way.
\end{remark}

\begin{example}\label{ex: exp}
	For the case $\varphi = \psi$, we have that $C_{\varphi, \varphi}:\mathscr{S}(\R) \to \mathscr{S}(\R)$ is continuous if and only if for every $\alpha\in \N_0$ and $p>0$ there exists $q>0$ such that
	$$ \sup_{x \in \R^d} \frac{(1+|x|)^p}{(1+|\varphi(x)|)^q} |\varphi^{(\alpha)}(x)| < \infty. $$
	Indeed, necessity follows from Corollary~\ref{alpha0cor} while sufficiency is due to Corollary~\ref{sufficient}. For $\varphi(x)=\psi(x)=\exp(x)$, we thus have that $C_{\exp,\exp}$ is continuous. Observe that neither the composition operator $C_{\exp}$ nor the multiplication operator $M_{\exp}$ acts on $\mathscr{S}(\R)$. We will show in Example~\ref{ex:power boundedness of exp} below that $C_{\exp,\exp}$ is even power bounded, and therefore uniformly mean ergodic, on $\mathscr{S}(\R)$.
\end{example}

\begin{proposition}
	Let $\varphi:\R^d\rightarrow\R^d$ be smooth such that for each $p>0$
	$$\sup_{x\in\R^d}\frac{(1+|\varphi(x)|)^p}{(1+|x|)}<\infty.$$
	Moreover, let $\psi\in  C^{\infty}(\R^d)$ be such that $C_{\psi,\varphi}$ acts on $\mathscr{S}(\R^d)$. Then $\psi\in \mathscr{S}(\R^d)$.		
\end{proposition}
\begin{proof}
	Let $\alpha\in\N_0^d$ and $p>0$. By applying Corollary~\ref{alpha0cor}  for $p+1$ we get $q>0$ and $C>0$ such that
	$$\sup_{x\in\R^d}(1+|x|)^p|\psi^{(\alpha)}(x)|\leq C\sup_{x\in\R^d}\frac{(1+|\varphi(x)|)^q}{(1+|x|)}.$$
	By hypothesis on $\varphi$, the supremum on the right is finite so the assertion follows. 
\end{proof}	


\begin{theorem}\label{theo: small growth}
	Let $\varphi:\R^d\rightarrow\R^d$ be smooth. Then, $C_{\psi,\varphi}$ acts on $\mathscr{S}(\R^d)$ for every $\psi\in\mathscr{S}(\R^d)$ if and only if $\partial_k\varphi_i\in\mathscr{O}_M(\R^d)$ for every $1\leq i,k\leq d$.
\end{theorem}

\begin{proof}
	Let $\partial_k\varphi_i\in\mathscr{O}_M(\R^d)$ for every $1\leq i,k\leq d$. By the multivariate Fa\`a di Bruno formula, $f\circ \varphi\in \mathscr{O}_M(\R^d)$ for every $f\in \mathscr{S}(\R^d)$. Hence $\psi\cdot (f\circ \varphi)\in \mathscr{S}(\R^d)$ whenever $\psi\in \mathscr{S}(\R^d)$.
	
	Now, let $C_{\psi,\varphi}$ act on $\mathscr{S}(\R^d)$ for every $\psi\in\mathscr{S}(\R^d)$. We assume that there are $1\leq i,k\leq d$ such that $\partial_k \varphi_i\notin \mathscr{O}_M(\R^d)$. Hence, there is $\alpha\in\N_0^d$ such that
	$$\forall\,j\in\N:\,\sup_{x\in\R^d}\frac{|\varphi_i^{(\alpha+e_k)}(x)|}{(1+|x|)^j}=\infty.$$ 
	Thus, there is a sequence $(x_j)_{j\in\N}$ such that $|x_{j+1}|>|x_j|+1$ as well as
	$$\frac{|\varphi_i^{(\alpha+e_k)}(x_j)|}{(1+|x_j|)^j}>j$$
	for each $j\in\N$.
	
	Now, we distinguish two cases. In case $(\varphi(x_j))_{j\in\N}$ is unbounded, by abuse of notation, we identify $(x_j)_{j\in\N}$ with a subsequence satisfying $|\varphi(x_{j+1})|>|\varphi(x_{j})|+1$, $j\in\N$. 
	Let $\varrho, g\in \mathscr{D}(B(0,1))$ be such that $\varrho=1$ in $B(0,1/2)$ and $g^{(e_i)}(0)=1$ as well as $g^{(\beta)}(0)=0$ for $\beta\in\N_0^d\backslash\{e_i\}$. 
	It is standard to show that
	\begin{equation}\label{eq: multiplier}
		\psi(x):=\sum_{j\in\N}\frac{\varrho(x-x_j)}{(1+|x_j|)^{j/2}}
	\end{equation}
	belongs to $\mathscr{S}(\R^d)$, as does 
	$$f(x)=\sum_{j\in\N} \frac{g(x-\varphi(x_j))}{(1+|x_j|)^{j/2}}.$$
	By the multivariate Fa\`a di Bruno formula, for $j\in\N$ we conclude
	\begin{eqnarray}\label{eq: f in S but Cf not}
		|\left(\psi\cdot (f\circ\varphi)\right)^{(\alpha+e_k)}(x_j)|&=&\left|\psi(x_j)\left(f\circ\varphi\right)^{(\alpha+e_k)}(x_j)\right|\nonumber\\
		&=&\left|\psi(x_j)\frac{\partial f}{\partial x_i}\left(\varphi(x_j)\right)\varphi_i^{(\alpha+e_k)}(x_j)\right|\nonumber\\
		&=&\frac{|\varphi_i^{(\alpha+e_k)}(x_j)|}{(1+|x_j|)^{j}}>j,
	\end{eqnarray}
	so that $\psi\cdot(f\circ \varphi)\notin\mathscr{S}(\R^d)$ contradicting the hypothesis that $C_{\psi,\varphi}$ acts on $\mathscr{S}(\R^d)$.
	
	To finish the proof, we consider the case when $(\varphi(x_j))_{j\in\N}$ is bounded. We define $\psi$ as in \eqref{eq: multiplier} and we consider a function $f\in \mathscr{D}(\R^d)$ such that $f(x)=x_i$ on a neighborhood of $\{\varphi(x_j):j\in\N\}$. Then \eqref{eq: f in S but Cf not} holds which again gives the desired contradiction.
\end{proof}

For smooth and bounded $\varphi:\R^d\rightarrow\R^d$, by Corollary \ref{alpha0cor}, $\psi\in \mathscr{S}(\R^d)$ is a necessary condition for $C_{\psi,\varphi}$ to act on $\mathscr{S}(\R^d)$. Therefore, as an immediate corollary of Theorem \ref{theo: small growth} we have the following.

\begin{corollary}
	Let $\varphi:\R^d\rightarrow\R^d$ be smooth and bounded. Then, $C_{\psi,\varphi}$ acts on $\mathscr{S}(\R^d)$ for every $\psi\in\mathscr{S}(\R^d)$ if and only if $\varphi\in\mathscr{O}_M(\R^d)$.
\end{corollary}

\section{Small decay multipliers}\label{sec:small decay multiplies}

In this section, we shall prove that for a family of weighted composition operators, including composition operators, the sufficient conditions from Corollary~\ref{sufficient} are also necessary for $C_{\psi,\varphi}$ to act on $\mathscr{S}(\R^d)$. In order to do so, we first prove the following result which will also be of use in section~\ref{sec: power boundedness} below.

\begin{lemma}\label{techlemma2}
	Let $I$ be a non-empty set and let $\varphi_i: \mathbb{R}^d \to \mathbb{R}^d$ be smooth mappings, $i\in I$. Moreover, let $\psi_i \in C^{\infty}(\mathbb{R}^d)$, $i\in I$, be such that for some $m>0$
	\begin{equation}\label{eq: techlemma2, assumption on psi}\inf_{i\in I}\inf_{|x|\geq m} (1+|x|)^m (1+|\varphi_i(x)|)^m |\psi_i(x)| > 0.
	\end{equation}
	Assume that there are $\alpha \in \mathbb{N}_0^d$ and $p>0$ such that
	\begin{equation}\label{eq: techlemma2, assumption 1}
		\sup_{i \in I} \sup_{x \in \mathbb{R}^d} \frac{(1+|x|)^p}{(1+|\varphi_i(x)|)^q} |\psi_i^{(\alpha)}(x)| = \infty
	\end{equation}
	for all $q>0$, or that there is $\alpha \in \mathbb{N}_0^d$ such that
	\begin{equation}\label{eq: techlemma2, assumotion 2}
		\sup_{i \in I} \sup_{x \in \mathbb{R}^d} \frac{1}{(1+|\varphi_i(x)|)^q} |\varphi_i^{(\alpha)}(x)| = \infty
	\end{equation}
	for all $q>0$. Then there is $f \in \mathscr{S}(\mathbb{R}^d)$ satisfying
	$$ \sup_{i \in I} \sup_{x\in \mathbb{R}^d} (1+|x|)^{\max\{p,m+1\}} \max\{ |(\psi_i\cdot (f\circ \varphi_i))(x)|, |(\psi_i\cdot (f\circ \varphi_i))^{(\alpha)}(x)|\} = \infty. $$
\end{lemma}
\begin{proof}
	If there are $\alpha\in\N_0^d$ and $p>0$ satisfying \eqref{eq: techlemma2, assumption 1} for every $q>0$, then the result follows by Lemma~\ref{techlemma} with $\lambda=0$.
	
	Thus, in order to complete the proof, we may assume that for every $\beta\in\N_0^d$ and $p>0$ there is $q(\beta,p)>0$ such that
	\begin{equation}\label{eq: techlemma2, aux1}
		\sup_{i \in I} \sup_{x \in \mathbb{R}^d} \frac{(1+|x|)^p}{(1+|\varphi_i(x)|)^{q(\beta,p)}} |\psi_i^{(\beta)}(x)| < \infty.
	\end{equation}
	We now fix $\alpha\in\N_0^d$ with minimal $|\alpha|$ such that \eqref{eq: techlemma2, assumotion 2} holds for every $q>0$. Clearly, $|\alpha|\geq 1$, and for every $\beta\in\N_0^d$ with $\beta\leq\alpha, \beta\neq 0$, there is $q'(\beta)>0$ such that
	\begin{equation}\label{eq: techlemma2, aux2}
		\sup_{i \in I} \sup_{x \in \mathbb{R}^d} \frac{1}{(1+|\varphi_i(x)|)^{q'(\beta)}} |\varphi_i^{(\alpha-\beta)}(x)| < \infty.
	\end{equation}
	With $m$ as in \eqref{eq: techlemma2, assumption on psi}, we choose sequences $(i_j)_{j\in\N}$, $(l_j)_{j\in\N}$, and $(x_j)_{j\in\N}$ in $I$, $\mathbb{N}$ and $\mathbb{R}^d$, respectively, such that $l_{j+1} > l_j > m+\max\{q(\beta,m):\beta\leq\alpha\}+\max\{q'(\beta):0\neq\beta\leq \alpha\}$, $|x_{j+1}| > 1 + |x_j| > 1+m$, and $|\varphi_{i_j}^{(\alpha)}(x_j)| > j(1+|\varphi_{i_j}(x_j)|)^{l_j}$, $j\in\mathbb{N}$.
	
	Next, we distinguish two cases. Suppose that $(|\varphi_{i_j}(x_j)|)_{j\in\N}$ is bounded. Take $f \in \mathscr{D}(\mathbb{R}^d)$ such that $f \equiv 1$ on a neighbourhood of $\{ \varphi_{i_j}(x_j) : j \in \mathbb{N}\}$. By \eqref{eq: techlemma2, assumption on psi} and the boundedness of $(|\varphi_{i_j}(x_j)|)_{j\in \N}$ we conclude $\inf_{j\in\N} (1+|x_j|)^m |\psi_{i_j}(x_j)| >0$ which implies
	$$ \sup_{i \in I} \sup_{x \in \mathbb{R}^d} (1+|x|)^{m+1} |(\psi_i\cdot (f\circ \varphi_i))(x)| \geq \sup_{j \in \mathbb{N}} (1+|x_j|)^{m+1} |\psi_{i_j}(x_j)| |f(\varphi_{i_j}(x_j))| = \infty$$
	proving the assertion in this case.
	
	Now, suppose that $(|\varphi_{i_j}(x_j)|)_{j\in\N}$ is unbounded. By identifying $(\varphi_{i_j}(x_j))_{j\in\N}$ with a subsequence, we assume $|\varphi_{i_{j+1}}(x_{j+1})| > 1+|\varphi_{i_j}(x_j)|$, $j\in\N$. Since $|\varphi_{i_j}^{(\alpha)}(x_j)| > j(1+|\varphi_{i_j}(x_j)|)^{l_j}$ there is $1 \leq k \leq d$ such that the $k$-th component $\varphi_{i_j,k}$ of $\varphi_{i_j}$ satisfies $|\varphi^{(\alpha)}_{i_j,k}(x_j)| > \frac{j}{\sqrt{d}}(1+|\varphi_{i_j}(x_j)|)^{l_j}$, $j \in \mathbb{N}$.  Let $g \in \mathscr{D}(B(0,1))$ be such that $g(x)=x_k$ in a neighbourhood of the origin. It is then standard to show that the function
	$$f(x) = \sum_{j \in \mathbb{N}} \frac{g(x-\varphi_{i_j}(x_j))}{(1+|\varphi_{i_j}(x_j)|)^{l_j-m}}$$
	belongs to $\mathscr{S}(\mathbb{R}^d)$. Since for $\beta\in\N_0^d$ with $\beta\leq \alpha, \beta\neq \alpha$, by the choice of $g$ and the multivariate Fa\`a di Bruno formula, we have
	$$\left(f\circ\varphi_{i_j}\right)^{(\alpha-\beta)}(x_j)=\frac{\varphi_{i_j,k}^{(\alpha-\beta)}(x_j)}{(1+|\varphi_{i_j}(x_j)|)^{l_j-m}},$$
	while obviously, $(f\circ\varphi_{i_j})(x_j)=0$. Hence, for $0\neq\beta\leq\alpha$ we obtain
	\begin{eqnarray*}
		&&\sup_{j\in\N}(1+|x_j|)^m|\psi_{i_j}^{(\beta)}(x_j)\left(f\circ\varphi_{i_j}\right)^{(\alpha-\beta)}(x_j)|\\
		&\leq&\sup_{j\in\N}(1+|x_j|)^m |\psi_{i_j}^{(\beta)}(x_j)|\frac{|\varphi_{i_j}^{(\alpha-\beta)}(x_j)|}{(1+|\varphi_{i_j}(x_j)|)^{l_j-m}}\\
		&\leq& \sup_{j\in\N}\frac{(1+|x_j|)^m}{(1+|\varphi_{i_j}(x_j)|)^{q(\beta,m)}}|\psi_{i_j}^{(\beta)}(x_j)|\frac{|\varphi_{i_j}^{(\alpha-\beta)}(x_j)|}{(1+|\varphi_{i_j}(x_j)|)^{q'(\beta)}}<\infty,
	\end{eqnarray*}
	where we have used $l_j-m>\max\{q(\beta,m):\beta\leq\alpha\}+\max\{q'(\beta):0\neq\beta\leq \alpha\}$ together with \eqref{eq: techlemma2, aux1} and \eqref{eq: techlemma2, aux2}.
	
	On the other hand,
	\begin{eqnarray*}
		(1+|x_j|)^m|\psi_{i_j}(x_j) (f \circ \varphi_{i_j})^{(\alpha)}(x_j)|&=& (1+|x_j|)^m(1+|\varphi_{i_j}(x_j)|)^m|\psi_{i_j}(x_j)| \frac{|\varphi^{(\alpha)}_{i_j,k}(x_j)|}{(1+|\varphi_{i_j}(x_j)|)^{l_j}}\\
		&\geq&(1+|x_j|)^m(1+|\varphi_{i_j}(x_j)|)^m|\psi_{i_j}(x_j)|\frac{j}{\sqrt{d}}.
	\end{eqnarray*}
	Combining the above with \eqref{eq: techlemma2, assumption on psi}, the boundedness of $\left((1+|x_j|)^m|\psi_{i_j}^{(\beta)}(x_j)\left(f\circ\varphi_{i_j}\right)^{(\alpha-\beta)}(x_j)|\right)_{j\in\N}$ for $0\neq\beta\leq\alpha$, and Leibniz' rule finally give 
	$$\sup_{i \in I} \sup_{x \in \mathbb{R}^d} (1+|x|)^m |(\psi_i(f \circ \varphi_i))^{(\alpha)}(x)| = \infty, $$
	which completes the proof.
\end{proof}

\begin{definition}
	For a smooth mapping $\varphi:\R^d\rightarrow\R^d$ a function $\psi\in C^\infty(\R^d)$ is said to be of {\em small decay with respect to} $\varphi$ if there is $m>0$ such that
	$$\inf_{|x|\geq m}(1+|x|)^m(1+|\varphi(x)|)^m|\psi(x)|>0.$$
\end{definition}

As an immediate consequence of the previous lemma we obtain the following result. 

\begin{theorem}\label{theo: small decay}
	Let $\varphi:\R^d\rightarrow\R^d$ be smooth and let $\psi\in C^\infty(\R^d)$ be of small decay with respect to $\varphi$. Then $C_{\psi,\varphi}$ acts on $\mathscr{S}(\R^d)$ if and only if	for each $\alpha\in\N_0^d$ and $p>0$ there is $q>0$ such that 
	$$\sup_{x\in\R^d}\frac{(1+|x|)^p}{(1+|\varphi(x)|)^q}|\psi^{(\alpha)}(x)|<\infty\quad\text{and}\quad\sup_{x\in\R^d}\frac{1}{(1+|\varphi(x)|)^q}|\varphi^{(\alpha)}(x)|<\infty.$$
\end{theorem}
\begin{proof}
	By Corollary~\ref{sufficient} the conditions are sufficient for $C_{\psi,\varphi}$ to act on $\mathscr{S}(\R^d)$ while Lemma~\ref{techlemma2} also implies necessity.
\end{proof}

\section{Power boundedness and topologizability}\label{sec: power boundedness}

In this section we study power boundedness and related properties for weighted composition operators $C_{\psi,\varphi}$ on $\mathscr{S}(\R^d)$. For $p>0$ and $\alpha\in\N_0^d$, obviously
$$\forall\,f\in\mathscr{S}(\R^d):\,\|f\|_{p,\alpha}=\sup_{x\in\R^d}(1+|x|)^p|f^{(\alpha)}(x)|$$
defines a continuous norm $\|\cdot\|_{p,\alpha}$ on $\mathscr{S}(\R^d)$ and the set of norms $\{\|\cdot\|_{p,\alpha}:\,p>0, \alpha\in\N_0^d\}$ defines the standard topology on $\mathscr{S}(\R^d)$. Whenever we equip the vector space $\mathscr{S}(\R^d)$ only with the norm $\|\cdot\|_{p,\alpha}$ we write $\mathscr{S}_{p,\alpha}(\R^d)$. Obviously, every continuous linear mapping from $\mathscr{S}(\R^d)$ into itself is in particular a continuous linear mapping from $\mathscr{S}(\R^d)$ into $\mathscr{S}_{p,\alpha}(\R^d)$.

Next, for a smooth mapping $\varphi:\R^d\rightarrow\R^d$ and $n\in\N$ we set $\varphi_n:=\varphi\circ\cdots\circ\varphi$ the $n$-th iteration of $\varphi$. Additionally, we set $\varphi_0(x)=x$ and for $\psi\in C^\infty(\R^d)$ and $n\in\N$ we define $$\psi^{n,\varphi}(x):=\prod_{j=1}^n\psi(\varphi_{j-1}(x)),\quad x\in\R^d.$$ In particular, $\varphi_1=\varphi$ and $\psi^{1,\varphi}=\psi$.

\begin{lemma}\label{theo: equicontinuity of weighted composition operators}
	Let $I$ be a non-empty set, $(\psi_i)_{i\in I}\in C^\infty(\R^d)^I$ and let $\varphi_i:\R^d\rightarrow\R^d$ be smooth mappings, $i\in I$, such that the weighted composition operators $C_{\psi_i,\varphi_i}$ act on $\mathscr{S}(\R^d)$, $i\in I$.
	\begin{itemize}
		\item[(a)] Let $\alpha\in\N_0^d$ and $p>0$ be fixed. Then, the set of continuous linear mappings $\{C_{\psi_i,\varphi_i}:\,i\in I\}$ is equicontinuous from $\mathscr{S}(\R^d)$ into $\mathscr{S}_{p,\alpha}(\R^d)$ if and only if for every $\lambda\in \N_0^d$ with $|\lambda|\leq |\alpha|$ there is $q>0$ such that
		\begin{equation}\label{eq: equicontinuity}
			\sup_{i\in I}\sup_{x\in\R^d}\frac{(1+|x|)^p}{(1+|\varphi_i(x)|)^q}\left|\Fmulti(x)\right|<\infty.
		\end{equation}
		\item[(b)] The set of continuous linear mappings $\{C_{\psi_i,\varphi_i}:\,i\in I\}$ is equicontinuous on $\mathscr{S}(\R^d)$ if and only if for every $\alpha,\lambda\in \N_0^d$ with $|\lambda|\leq |\alpha|$ and each $p>0$ there is $q>0$ such that \eqref{eq: equicontinuity} holds.
	\end{itemize}
\end{lemma}

\begin{proof}
	Clearly, (b) is a direct consequence of (a). In order to prove (a), since $\mathscr{S}(\R^d)$ is a Fr\'echet space, by the Uniform Boundedness Principle equicontinuity of $\{C_{\psi_i,\varphi_i}:\,i\in I\}$ from $\mathscr{S}(\R^d)$ into $\mathscr{S}_{p,\alpha}(\R^d)$ is equivalent to the boundedness of $\{C_{\psi_i,\varphi_i} f: i\in I\}$ in $\mathscr{S}_{p,\alpha}(\R^d)$ for every $f\in \mathscr{S}(\R^d)$. Thus, by Lemma~\ref{techlemma} and \eqref{eq: Leibniz and Faa di Bruno}, we obtain the claimed equivalence.
\end{proof}

\begin{theorem}\label{cor: power boundedness characterized}
	Let $\varphi:\R^d\rightarrow\R^d$ be a smooth mapping and let $\psi\in C^\infty(\R^d)$. Then, the following are equivalent.
	\begin{itemize}
		\item[(i)] $C_{\psi,\varphi}$ acts on $\mathscr{S}(\R^d)$ and is power bounded.
		\item[(ii)] For all $\alpha,\lambda\in \N_0^d$ with $|\lambda|\leq |\alpha|$ and each $p>0$ there exists $q>0$ such that
		$$\sup_{n\in\N}\sup_{x\in\R^d}\frac{(1+|x|)^p}{(1+|\varphi_n(x)|)^q}\left|\Fmultiterate(x)\right|<\infty.$$
	\end{itemize}
\end{theorem}

\begin{proof}
	Since $C_{\psi,\varphi}^n f=C_{\psi^{n,\varphi},\varphi_n}f$ for every $f\in \mathscr{S}(\R^d)$, $n\in\N$, the assertion follows immediately from Theorem~\ref{theo: weighted compositions characterized} and Theorem~\ref{theo: equicontinuity of weighted composition operators} (b).
\end{proof}

\begin{remark}\label{rem: power bounded multiplication operators}
	Using the same arguments as in Remark~\ref{rem: multiplication}, by Theorem~\ref{cor: power boundedness characterized}, for $\psi\in\mathscr{O}_M(\R^d)$ the corresponding multiplication operator $M_\psi$ is power bounded on $\mathscr{S}(\R^d)$ if and only if for every $\gamma\in\N_0^d$ there is $r>0$ such that
	$$\sup_{n\in\N}\sup_{x\in\R^d}(1+|x|)^{-r}\left|\left(\psi^n\right)^{(\gamma)}(x)\right|<\infty,$$
	i.e.\ the sequence $(\psi^n)_{n\in\N}$ is bounded in $\mathscr{O}_M(\R^d)$ (cf.\ \cite[Theorem 4.3]{AlMe22}).
\end{remark}

Using Remark~\ref{rem: one dimensional} as in the proof of Corollary~\ref{cor: one dimensional weighted composition characterized} one obtains the following result.

\begin{corollary}\label{cor: one dimensional power boundedness characterized}
	Let $\psi,\varphi\in C^\infty(\R)$, $\varphi$ be real valued. Then, the following are equivalent.
	\begin{itemize}
		\item[(i)] $C_{\psi,\varphi}$ acts on $\mathscr{S}(\R)$ and is power bounded.
		\item[(ii)] For all $\alpha,\beta,\lambda\in\N_0$ with $\alpha\geq\beta\geq\lambda$ and for each $p>0$ there is $q>0$ such that
		$$\sup_{n\in\N}\sup_{x\in\R}\frac{(1+|x|)^p}{(1+|\varphi_n(x)|)^q}\left|\left(\psi^{n,\varphi}\right)^{(\alpha-\beta)}(x)B_{\beta,\lambda}\left(\varphi_n'(x),\ldots,\varphi_n^{(\beta-\lambda+1)}(x)\right)\right|<\infty.$$
	\end{itemize}
\end{corollary}

The next theorem follows immediately from Lemma~\ref{techlemma2} and Theorem~\ref{theo: small decay}.

\begin{theorem}\label{CorollaryPowerBoundednessSmallDecay}
	Let $\varphi: \mathbb{R}^d \to \mathbb{R}^d$ be a smooth mapping and let $\psi \in C^{\infty}(\mathbb{R}^d)$ be of small decay with respect to $\varphi$.
	Then, the following are equivalent.
	\begin{enumerate}
		\item[(i)] $C_{\psi,\varphi}$ acts on $\mathscr{S}(\R^d)$ and is power bounded.
		\item[(ii)] The following two conditions are satisfied:
		\begin{itemize}
			\item[(a)] For all $p>0$ and $\alpha \in \mathbb{N}_0^d$ there exists $q>0$ such that
			$$ \sup_{n \in \mathbb{N}} \sup_{x \in \mathbb{R}^d} \frac{(1+|x|)^p}{(1+|\varphi_n(x)|)^q} |(\psi^{n, \varphi})^{(\alpha)}(x)| < \infty. $$
			\item[(b)] For all $\alpha \in \mathbb{N}_0^d$ there exists $q>0$ such that
			$$ \sup_{n \in \mathbb{N}} \sup_{x \in \mathbb{R}^d} \frac{1}{(1+|\varphi_n(x)|)^q} |\varphi_n^{(\alpha)}(x)| < \infty. $$
		\end{itemize}
	\end{enumerate}
\end{theorem}
\begin{remark}
	\label{sufficientpb}
	We observe that, arguing as in Corollary~\ref{alpha0cor} and Corollary~\ref{sufficient}, the hypothesis for $\psi$ to be of small decay with respect to $\varphi$ in Theorem~\ref{CorollaryPowerBoundednessSmallDecay} is only needed for the necessity of (b) in (ii). Condition (a) is necessary for the power boundedness of $C_{\psi,\varphi}$ while (a) and (b) together are always sufficient conditions for the power boundedness of $C_{\psi,\varphi}$ whether $\psi$ is of small decay with respect to $\varphi$ or not.
\end{remark}

\begin{remark}
	\label{powerboundedcomposition}
	Obviously, $\psi(x)\equiv 1$ is of small decay with respect to every smooth $\varphi:\R^d\rightarrow\R^d$ so that Theorem~\ref{CorollaryPowerBoundednessSmallDecay} is applicable to composition operators. Employing the same arguments as in Remark~\ref{rem: one dimensional composition operators}, we derive easily that the composition operator $C_\varphi$ acts on $\mathscr{S}(\R^d)$ and is power bounded if and only if it satisfies the two conditions
	\begin{itemize}
		\item[(a)] there are $k,l>0$ such that for every $n\in\N$ it holds $|\varphi_n(x)|\geq |x|^{k}$ whenever $|x|\geq l$, and
		\item[(b)] for every $\alpha\in \N_0$ there are $C,q>0$ such that  $|\varphi_n^{(\alpha)}(x)|\leq C (1+|\varphi_n(x)|)^q$ for every $n\in\N, x\in\R^d$.
	\end{itemize}
	For $d=1$, this characterization has been obtained in \cite[Proposition 3.9]{FeGaJo18}. We observe, that when $d=1$, then any polynomial $\varphi$ with $\text{deg}(\varphi)\geq 2$ satisfies condition (a) above. Indeed, this is easily obtained from $\lim_{|x|\to\infty}\frac{|\varphi(x)|}{|x|}=\infty$. Combining this observation with \cite[Theorem 3.11]{FeGaJo18} we obtain that condition (b) above is satisfied if and only if $\deg(\varphi)\geq 2$ and $\varphi$ has no fixed points (and hence the degree of $\varphi$ is even).
\end{remark}

\begin{remark}\label{rem:mean ergodicity}
	Because $\mathscr{S}(\R^d)$ is a Montel space, \cite[Proposition 3.3]{BoPaRi11} combined with \cite[Theorem 2.5]{KaSa22} yield that $C_{\psi,\varphi}$ as well as its transpose are uniformly mean ergodic whenever $C_{\psi,\varphi}$ is power bounded on $\mathscr{S}(\R^d)$.
\end{remark}

\begin{example}
	\label{ex:power boundedness of exp}
	In Example \ref{ex: exp} we have already seen that for $\psi=\varphi=\exp$ the weighted composition operator $C_{\exp,\exp}$ acts on $\mathscr{S}(\R)$ although neither $\exp$ is a symbol for $\mathscr{S}(\R)$ nor $\exp\in \mathscr{O}_M(\R)$. We will now show that $C_{\exp,\exp}$ is even power bounded on $\mathscr{S}(\R)$, and thus uniformly mean ergodic, by Remark \ref{rem:mean ergodicity}. To do so, we first notice that
	\begin{equation}\label{eqPBBexpx2}
		\left(\varphi_n\right)'(x)=\left(\varphi'\right)^{n,\varphi}(x)=\varphi^{n,\varphi}(x)=\varphi_1(x)\varphi_2(x)\cdots\varphi_n(x).
	\end{equation}
	Thus, by Corollary \ref{cor: one dimensional power boundedness characterized}, $C_{\exp,\exp}$ is power bounded on $\mathscr{S}(\R)$ if for every $p>0$ and $\alpha\in\N_0$ there is $q>0$ such that
	\begin{equation}\label{eq: characteristic for exp}
		\sup_{n \in \mathbb{N}} \sup_{x \in \mathbb{R}} \frac{(1+|x|)^p}{(1+|\varphi_n(x)|)^q} |(\varphi^{n,\varphi})^{(\alpha)}(x)| < \infty
	\end{equation}
	holds.
	
	In order to prove \eqref{eq: characteristic for exp}, we show as a first step the following auxiliary inequality. 
	\begin{equation}\label{EqPBBexpx}
		\forall\,\alpha\in\N_0\,\exists\,n_\alpha\in\N\,\forall\,n\geq n_\alpha, x\in\R:\,(\varphi^{n,\varphi})^{(\alpha)}(x) \leq \varphi_1(x) (\varphi_{n}(x))^{2+\alpha}.
	\end{equation}
	In order to prove that~\eqref{EqPBBexpx} indeed holds true, we note that due to $\varphi^{2,\varphi}=\varphi_1\varphi_2$ the inequality~\eqref{EqPBBexpx} holds for $\alpha=0$ and $n=2$. Now, assume that the inequality~\eqref{EqPBBexpx} holds for $\alpha=0$ and some $n\geq 2$. Then, since $(\varphi_n(x))^2 \leq \varphi_{n+1}(x)$ for every $x$,
	$$\varphi^{n+1,\varphi}(x) = \varphi^{n,\varphi}(x) \varphi_{n+1}(x) \leq \varphi_1(x) (\varphi_n(x))^2 \varphi_{n+1}(x) \leq \varphi_1(x)(\varphi_{n+1}(x))^2$$
	so that~\eqref{EqPBBexpx} also holds for $\alpha=0$ and $n+1$ proving the validity of~\eqref{EqPBBexpx} for $\alpha=0$ with $n_0=2$.
	
	In order to establish~\eqref{EqPBBexpx} for $\alpha=1$ we first provide some auxiliary estimates. As $\varphi(x)=e^x$, for every $n\geq1$,
	$$\varphi_n(x) = e^{\varphi_{n-1}(x)} = \sum_{j=0}^{\infty} \frac{(\varphi_{n-1}(x))^j}{j!},$$
	so that
	\begin{equation}\label{EqPBBexpx4}
		\forall\,j\in\N_0,n\in\N, n\geq 2,x\in\R:\,(\varphi_{n-1}(x))^j \leq j! \varphi_n(x).
	\end{equation}
	Besides that, since for $n\geq 1$ it holds $\varphi_n(\mathbb{R}) = (\varphi_{n-1}(0), +\infty)$
	we also have
	\begin{equation}\label{EqPBBexpx3}
		\forall\,n\geq 5, x\in\R:\,n \leq \varphi_{n-2}(0) \leq \varphi_{n-1}(x).
	\end{equation}
	For arbitrary $C>0$ and $j\in\N$, choosing $N\geq 5$ so large that $C(j+2)!\leq N$, by~\eqref{EqPBBexpx3} and~\eqref{EqPBBexpx4}, it holds for all $x \in \mathbb{R}$ and $n \geq N$,
	\begin{equation}\label{EqPBBexpx5}
		nC(\varphi_{n-1}(x))^j = n (C(j+2)!) \frac{(\varphi_{n-1}(x))^j}{(j+2)!} \leq (\varphi_{n-1}(x))^2 \frac{(\varphi_{n-1}(x))^j}{(j+2)!} \leq \varphi_n(x).
	\end{equation}
	
	To show~\eqref{EqPBBexpx} for $\alpha=1$, we introduce $\Psi_n(x) := 1+\varphi^{1,\varphi}(x) + \cdots + \varphi^{n-1,\varphi}(x)$. Applying \eqref{eqPBBexpx2} we then obtain
	\begin{align}\nonumber
		(\varphi^{n,\varphi})'(x) &= (\varphi_1 \cdots \varphi_n)'(x) = (\varphi_1)'(x) \cdot \varphi_2(x) \cdots \varphi_{n}(x) + \cdots + \varphi_1(x) \cdots \varphi_{n-1}(x) \cdot (\varphi_n)'(x) \\
		&= (\varphi_1(x) \cdots \varphi_n(x))(1+\varphi_1(x) + \cdots + \varphi_1(x) \cdots \varphi_{n-1}(x))\nonumber\\
		&=\varphi^{n,\varphi}(x)\Psi_n(x).\label{EqPBBexpx6}
	\end{align}
	By inequality \eqref{EqPBBexpx4} for $j=0$, $\varphi_n(x)\geq 1$ for every $n\geq 2, x\in\R,$ implying $\varphi^{n,\varphi}(x)\leq \varphi^{n+1,\varphi}(x)$ for $n\geq 1, x\in\R$. Hence, $\Psi_n(x)\leq 1+(n-1)\varphi^{n-1,\varphi}(x), x\in\R, n\in\N$. Applying~\eqref{EqPBBexpx3},~\eqref{EqPBBexpx} for $\alpha=0$ and~\eqref{EqPBBexpx5} for $C=2, j=3$ there is $N\geq 2\cdot 5!$ such that for all $n\geq N$ and all $x \in \mathbb{R}$,
	\begin{align*}
		\Psi_n(x) &\leq 1+(n-1)\varphi^{n-1,\varphi}(x) \leq 1+ \varphi_{n-1}(x) \varphi^{n-1,\varphi}(x) \leq 1+(\varphi_{n-1}(x))^3\\
		&\leq 2(\varphi_{n-1}(x))^3 \leq \varphi_n(x).
	\end{align*}
	Hence, by \eqref{EqPBBexpx6} and by \eqref{EqPBBexpx} for $\alpha=0$, for all $n\geq N$ and all $x \in \mathbb{R}$,
	$$ (\varphi^{n,\varphi})'(x) = \varphi^{n,\varphi}(x) \Psi_n(x) \leq \varphi_1(x)(\varphi_n(x))^3$$
	which proves~\eqref{EqPBBexpx} for $\alpha=1$.
	
	Now, fix $\alpha \in \mathbb{N}$ and assume that for all $\beta < \alpha$ there exists $n_{\beta} \in \mathbb{N}$ (without losing generality, $n_{\beta} < n_{\gamma}$ if $\beta \leq \gamma$ and $\beta\neq \gamma$) such that~\eqref{EqPBBexpx} is satisfied, i.e.,
	\begin{equation}\label{EqPBBexpxHI}
		(\varphi^{n,\varphi})^{(\beta)}(x) \leq \varphi_1(x)(\varphi_n(x))^{2+\beta}, \qquad x \in \mathbb{R}, \ n \geq n_{\beta}.
	\end{equation}
	We show~\eqref{EqPBBexpx} for $\alpha$ and $n_{\alpha}$, where $n_{\alpha} \in \mathbb{N}$ satisfies $n_\alpha\geq \max\{n_\beta; \beta\leq\alpha, \beta\neq \alpha\}+1$ and is such that for every $x \in \mathbb{R}$ and $n\geq n_{\alpha}$
	\begin{equation}\label{EqPBBexpx10}
		2^{\alpha-1}n (\varphi_{n-1}(x))^{\alpha+2} \leq \varphi_n(x),
	\end{equation}
	which is possible by~\eqref{EqPBBexpx5}. From~\eqref{EqPBBexpx6} and Leibniz' rule, for every $n\geq n_\alpha$
	\begin{equation}\label{EqPBBexpx7}
		(\varphi^{n,\varphi})^{(\alpha)}(x) = (\varphi^{n,\varphi} \Psi_n)^{(\alpha-1)}(x) = \sum_{\beta \leq \alpha-1} \binom{\alpha-1}{\beta} (\varphi^{n,\varphi})^{(\beta)}(x) (\Psi_n)^{(\alpha-1-\beta)}(x).
	\end{equation}
	Moreover, by Leibniz' rule, for all $\gamma \in \mathbb{N}_0$, $j \in \mathbb{N}$, and $x \in \mathbb{R}$,
	\begin{align*}
		(\varphi^{j+1,\varphi})^{(\gamma)}(x) &= (\varphi^{j,\varphi})^{(\gamma)}(x) 	\varphi_{j+1}(x) + \sum_{\widetilde{\gamma}<\gamma} \binom{\gamma}{\widetilde{\gamma}} (\varphi^{j,\varphi})^{(\widetilde{\gamma})}(x) (\varphi_{j+1})^{(\gamma-\widetilde{\gamma})}(x) \\
		&\geq (\varphi^{j,\varphi})^{(\gamma)}(x)\varphi_j(0) \geq (\varphi^{j,\varphi})^{(\gamma)}(x).
	\end{align*}
	Applying the above inequality for $j=1,\ldots,n-2$, by the choice of $n_{\alpha}$, using inequality~\eqref{EqPBBexpxHI}, for every $n\geq n_\alpha$, and taking in account $\varphi_{n-1}(x)>1$ for all $x\in\R$, we get for each $\beta\leq \alpha-1$
	\begin{align}\label{EqPBBexpx8}
		(\Psi_n)^{(\alpha-1-\beta)}(x) &= (1+\varphi^{1,\varphi}+\cdots+\varphi^{n-1,\varphi})^{(\alpha-1-\beta)}(x) \nonumber\\
		&\leq \frac{d^{\alpha-1-\beta}}{dx^{\alpha-1-\beta}}1+(n-1)(\varphi^{n-1,\varphi})^{(\alpha-1-\beta)}(x)\nonumber\\
		&\leq 1 + (n-1)\varphi_1(x)(\varphi_{n-1}(x))^{\alpha-\beta+1}\nonumber\\
		&\leq (1+\varphi_1(x))n (\varphi_{n-1}(x))^{\alpha-\beta+1}\nonumber\\
		&\leq \varphi_{n-1}(x)n (\varphi_{n-1}(x))^{\alpha-\beta+1}\nonumber\\
		&\leq n\varphi_{n-1}(x)^{\alpha+2}\nonumber\\
		&\leq \frac1{2^{\alpha-1}} \varphi_{n}(x).
	\end{align}
	Thus, combining~\eqref{EqPBBexpxHI} with~\eqref{EqPBBexpx7} and~\eqref{EqPBBexpx8} for $n\geq n_\alpha$, we conclude
	$$ (\varphi^{n, \varphi})^{(\alpha)}(x) \leq \frac1{2^{\alpha-1}} \sum_{\beta \leq \alpha-1} \binom{\alpha-1}{\beta} \varphi_1(x)(\varphi_n(x))^{2+\beta} \varphi_{n}(x) \leq \varphi_1(x) (\varphi_n(x))^{2+\alpha}, \qquad x \in \mathbb{R}.$$	
	The proof of \eqref{EqPBBexpx} is complete.
	
	In order to derive \eqref{eq: characteristic for exp} from \eqref{EqPBBexpx}, let $p>0$ and $\alpha\in\N_0$ be fixed. Set $q=p+3+\alpha$. Then, for $n\geq n_\alpha$ it follows for  $x\geq 0$
	$$\frac{(1+x)^p}{(1+\varphi_n(x))^{p+3+\alpha}} (\varphi^{n,\varphi})^{(\alpha)}(x)\leq\frac{(1+x)^p\varphi_1(x)}{(1+\varphi_n(x))^{p+1}}\leq 1$$
	while for $x<0$ we have
	$$\frac{(1+|x|)^p}{(1+\varphi_n(x))^{p+3+\alpha}} (\varphi^{n,\varphi})^{(\alpha)}(x)\leq (1+|x|)^p\varphi_1(x)\leq \sup_{y\geq 0}(1+y)^p\exp(-y)<\infty.$$
\end{example}

In the remainder of this section, we deal with ($m$-)topologizability of weighted composition operators on $\mathscr{S}(\R^d)$.

\begin{proposition}\label{cor: topologizability characterized}
	Let $\varphi:\R^d\rightarrow\R^d$ be a smooth mapping and let $\psi\in C^\infty(\R^d)$. Then, the following are equivalent.
	\begin{itemize}
		\item[(i)] $C_{\psi,\varphi}$ acts on $\mathscr{S}(\R^d)$ and is topologizable ($m$-topologizable).
		\item[(ii)] For all $\alpha,\lambda\in\N_0^d$ with $|\lambda|\leq |\alpha|$ and each $p>0$ there exists $q>0$ such that
		$$\forall\,n\in\N:\,\sup_{x\in\R^d}\frac{(1+|x|)^p}{(1+|\varphi_n(x)|)^q}\left|\Fmultiterate(x)\right|<\infty\quad(<M^n\text{ for some }M>0).$$
	\end{itemize}
\end{proposition}

\begin{proof}
	We only prove the assertion for topologizability. The proof for $m$-topologizability is done by the same argument with obvious modifications.
	
	By definition, $C_{\psi,\varphi}$ is topologizable on $\mathscr{S}(\R^d)$ if and only if for every $p>0$ and $\alpha\in\N_0^d$ there is a sequence $(a_n)_{n\in\N}$ in $(0,\infty)$ such that $\{a_n C_{\psi,\varphi}^n:\,n\in\N\}$ is equicontinuous from $\mathscr{S}(\R^d)$ into $\mathscr{S}_{p,\alpha}(\R^d)$. Since $a_nC_{\psi,\varphi}^n=C_{a_n\psi^{n,\varphi},\varphi_n}$, $n\in\N$, by Lemma~\ref{theo: equicontinuity of weighted composition operators} (a), $C_{\psi,\varphi}$ is topologizable precisely, when for every $\alpha\in\N_0^d$ and $p>0$ there is $(a_n)_{n\in\N}$ in $(0,\infty)$ such that for each $\lambda\in\N_0^d$ with $|\lambda|\leq |\alpha|$ there is $q>0$ with
	$$\sup_{n\in\N}\sup_{x\in\R^d}\frac{(1+|x|)^p}{(1+|\varphi_n(x)|)^q}\left|a_n\Fmultiterate(x)\right|<\infty,$$
	where we have used that $F_{\alpha,\lambda}^{\varphi_n, a_n \psi^{n,\varphi}}(x)=a_n\Fmultiterate(x)$, $x\in\R^d$. Since the number of $\lambda\in\N_0^d$ with $|\lambda|\leq |\alpha|$ is finite, the latter property is easily seen to be equivalent to (ii) which proves the Proposition.
\end{proof}

\begin{remark}\label{rem: topologizability}
	\begin{itemize}
		\item[(i)] Again, using Remark~\ref{rem: one dimensional} as in the proof of Corollary~\ref{cor: one dimensional weighted composition characterized}, for $d=1$ a weighted composition operator $C_{\psi,\varphi}$ on $\mathscr{S}(\R)$ is topologizable, respectively $m$-to\-po\-lo\-gi\-za\-ble, if and only if for all $\alpha,\beta,\lambda\in\N_0$ with $\alpha \geq \beta \geq \lambda$ and for each $p>0$ there is $q>0$ such that
		$$\forall\,n\in\N:\,\sup_{x\in\R}\frac{(1+|x|)^p}{(1+|\varphi_n(x)|)^q}\left|\left(\psi^{n,\varphi}\right)^{(\alpha-\beta)}(x)B_{\beta,\lambda}\left(\varphi_n'(x),\ldots,\varphi_n^{(\beta-\lambda+1)}(x)\right)\right|<\infty,$$
		respectively if and only if
		for all $\alpha,\beta,\lambda\in\N_0$ with $\alpha \geq \beta \geq \lambda$ and for each $p>0$ there are $M,q>0$ such that
		$$\forall\,n\in\N:\,\sup_{x\in\R}\frac{(1+|x|)^p}{(1+|\varphi_n(x)|)^q}\left|\left(\psi^{n,\varphi}\right)^{(\alpha-\beta)}(x)B_{\beta,\lambda}\left(\varphi_n'(x),\ldots,\varphi_n^{(\beta-\lambda+1)}(x)\right)\right|\leq M^n.$$ 
		For the special case $\psi=1$, similar as in Remark~\ref{rem: one dimensional composition operators}, the above characterizations of ($m$-)topologizablity simplify to the two conditions
		\begin{itemize}
			\item there is $k>0$ such that for each $n\in\N$ there is $r_n>0$ such that $|\varphi_n(x)|\geq |x|^{1/k}$ for all $|x|\geq r_n$, and
			\item for every $\alpha\in\N$ there is $q>0$ (resp.\ there are $M,q>0$) with $\sup_{x\in\R}\frac{|\varphi_n^{(\alpha)}(x)|}{(1+|\varphi_n(x)|)^q}<\infty$ (resp.\ $\sup_{x\in\R}\frac{|\varphi_n^{(\alpha)}(x)|}{(1+|\varphi_n(x)|)^q}\leq M^n$) for every $n\in\N$.
		\end{itemize}
		\item[(ii)] Similarly as in Remark~\ref{rem: power bounded multiplication operators}, by Proposition \ref{cor: topologizability characterized}, for $\psi\in\mathscr{O}_M(\R^d)$ the corresponding multiplication operator $M_\psi$ is ($m$-)topologizable on $\mathscr{S}(\R^d)$ if and only if for every $\gamma\in\N_0^d$ there is $r>0$ (resp.\ there are $M,r>0$) such that $$\forall\,n\in\N:\,\sup_{x\in\R^d}(1+|x|)^{-r}\left|\left(\psi^n\right)^{(\gamma)}(x)\right|<\infty,$$
		respectively
		$$\forall\,n\in\N:\,\sup_{x\in\R^d}(1+|x|)^{-r}\left|\left(\psi^n\right)^{(\gamma)}(x)\right|\leq M^n$$
	\end{itemize}
\end{remark}
Analogously to the proof of Theorem \ref{CorollaryPowerBoundednessSmallDecay}, we can characterize topologizability and $m$-to\-po\-lo\-gi\-za\-bi\-li\-ty in a much more operable way.

\begin{theorem}\label{CorollaryTopSmallDecay}
	Let $\varphi: \mathbb{R}^d \to \mathbb{R}^d$ be a smooth mapping and let $\psi \in C^{\infty}(\mathbb{R}^d)$ be of small decay with respect to $\varphi$.
	Then, the following are equivalent.
	\begin{enumerate}
		\item[(i)] $C_{\psi,\varphi}$ acts on $\mathscr{S}(\R^d)$ and is topologizable (m-topologizable) .
		\item[(ii)] The following two conditions are satisfied:
		\begin{itemize}
			\item[(a)] For all $p>0$ and $\alpha \in \mathbb{N}_0^d$ there exists $q>0$ such that, for all $n\in\N$
			$$  \sup_{x \in \mathbb{R}^d} \frac{(1+|x|)^p}{(1+|\varphi_n(x)|)^q} |(\psi^{n, \varphi})^{(\alpha)}(x)| < \infty \ (<M^n \text{ for some } M>0). $$
			\item[(b)]    For all $\alpha \in \mathbb{N}_0^d$ there exists $q>0$ such that, for all $n\in\N$
			$$ \sup_{x \in \mathbb{R}^d} \frac{1}{(1+|\varphi_n(x)|)^q} |\varphi_n^{(\alpha)}(x)| < \infty  \ (<M^n \text{ for some } M>0). $$
		\end{itemize}
	\end{enumerate}
\end{theorem}

\begin{remark}
	\label{remarktop}
	Condition (a) in Theorem \ref{CorollaryTopSmallDecay} is necessary for $C_{\psi,\varphi}$ to be topologizable ($m$-topologizable) without the assumption that $\psi$ is of small decay with respect to $\varphi$. Additionally, conditions (a) and (b) are always sufficient for topologizability ($m$-topologizability) of $C_{\psi,\varphi}$.
\end{remark}

\begin{example}
	\label{ex1}
	Fix $\psi \equiv 1$. From Remark~\ref{rem: topologizability}(i), observe that if $\varphi$ is a polynomial, then $C_{\varphi}:\mathscr{S}(\mathbb{R}) \to \mathscr{S}(\mathbb{R})$ is topologizable if and only if $C_{\varphi}$ is well defined if and only if $\deg(\varphi) \geq 1$. On the other hand, for $\varphi(x)=x+1$, $x\in\R$,  $n\in\N$, $p\geq 1$ it holds
	$$\sup_{x\in\R}\frac{(1+|x|)^p}{(1+|\varphi_n(x)|)^p}=\sup_{x\in\R}\left(\frac{1+|x|}{1+|x+n|}\right)^p= (1+|-n|)^p\leq (2^{p})^n. $$
	Hence condition (a) in Theorem \ref{CorollaryTopSmallDecay} is satisfied. Since condition (b) in Theorem \ref{CorollaryTopSmallDecay} is trivially satisfied, the corresponding composition operator, the so-called translation operator $C_{x+1}$, is $m$-topologizable but not power bounded, because by the same computation as above it does not satisfy condition (a) of Theorem \ref{CorollaryPowerBoundednessSmallDecay}, as observed in
	\cite[Remark 2]{FeGaJo18}.
	
\end{example}

\begin{example}
	Let $\varphi: \mathbb{R} \to \mathbb{R}$, $\varphi(x)=ax+b$, $a,b\in\R,$ $a\notin \{0,1\}$. Then $C_{\varphi}$ is $m$-topologizable. In fact, it follows from Theorem \ref{CorollaryTopSmallDecay} and 
	$$\lim_n \frac{1}{|a|^n}\sup_{x\in\R}\frac{1+|x|}{1+\left|a^nx+\frac{b(1-a^n)}{1-a}\right|}=\frac{|b|}{|a-1|}, \quad \text{if } |a|>1$$

	\noindent and 
	
	$$\lim_n |a|^n\sup_{x\in\R}\frac{1+|x|}{1+\left|a^nx+\frac{b(1-a^n)}{1-a}\right|}=\max\left\{1,\frac{|b|}{|1-a|}\right\}, \quad \text{if } 0<|a|<1.$$
	
	Finally, for $a=-1$, we have $C_\varphi^2=\text{id}_{\mathscr{S}(\R)}$.
\end{example}

\begin{example}\label{Example: sqrt of 1+x2}
	Let $\varphi(x)=\sqrt{1+x^2}$, and let $\psi$ be a fixed non-null polynomial, which is of small decay with respect to $\varphi$. Then $C_{\varphi}:\mathscr{S}(\mathbb{R}) \to \mathscr{S}(\mathbb{R})$ is power bounded by~\cite[Example 2]{FeGaJo18}. We show that for $C_{\psi,\varphi}:\mathscr{S}(\mathbb{R}) \to \mathscr{S}(\mathbb{R})$ topologizable and $m$-topologizability are equivalent and that this holds precisely when $\psi$ is constant. Moreover, $C_{\psi,\varphi}: \mathscr{S}(\mathbb{R}) \to \mathscr{S}(\mathbb{R})$ is power bounded if, and only, if $\psi \equiv c$ for some $|c| \leq 1$.
	
	Indeed, since $C_\varphi$ is power bounded, by Remark~\ref{powerboundedcomposition} and Theorem~\ref{CorollaryTopSmallDecay}, ($m$-)topologizability of $C_{\psi,\varphi}$ is equivalent to condition (a) from Theorem~\ref{CorollaryTopSmallDecay} (ii).
	
	Hence, if $\psi\equiv c\in\C$, condition (a) for $m$-topologizability from Theorem~\ref{CorollaryTopSmallDecay} (ii) holds.
	
	On the other hand, if $\deg(\psi) \geq 1$, for every $q \in \mathbb{N}$, we find $n=q+1$ so that
	$$ \sup_{x \in \mathbb{R}} \frac{|\psi^{n,\varphi}(x)|}{(1+|\varphi_n(x)|)^q} = \sup_{x \in \mathbb{R}} \frac{|\psi(x) \psi(\varphi(x)) \cdots \psi(\varphi_{n-1}(x))|}{(1+\sqrt{n+x^2})^q} = \infty, $$
	since the numerator has degree $n \cdot \deg(\psi)=(q+1) \cdot \deg(\psi)$ and the denominator, $q$. Therefore $C_{\psi,\varphi}$ is not topologizable (and not power bounded).

	Finally, fix $\psi \equiv c$. If $|c|\leq 1$,  then, $C_{\psi,\varphi}$ is power bounded. Otherwise, for all $q>0$,
	$$ \sup_{n \in \mathbb{N}} \sup_{x \in \mathbb{R}} \frac{|\psi^{n,\varphi}(x)|}{(1+|\varphi_n(x)|)^q} = \sup_{n \in \mathbb{N}} \sup_{x \in \mathbb{R}} \frac{|\psi(x) \psi(\varphi(x)) \cdots \psi(\varphi_{n-1}(x))|}{(1+\sqrt{n+x^2})^q} \geq \sup_{n \in \mathbb{N}} \frac{c^n}{(1+\sqrt{n})^q} = \infty, $$
	so $C_{\psi,\varphi}$ is not power bounded.
	
\end{example}

\section{Power boundedness of weighted composition operators for polynomials}\label{sec:polynomials}
\subsection{Power boundedness of $C_{\psi,\varphi}$ on $\mathscr{S}(\R)$ for $\varphi(x)=ax+b$.}\label{subsec:degree 1}

In this subsection we study univariate polynomials of degree one. For $a\in\R\backslash\{0\}, b\in\R$, as is well known, and may be checked by Remark~\ref{rem: one dimensional composition operators}, $\varphi(x)=ax+b$ is a symbol for $\mathscr{S}(\R)$. For the special case $\varphi(x)=x$ the composition operator $C_{\psi,\varphi}$ is the multiplication operator $M_\psi$. Power boundedness of multiplication operators in one variable has been characterized by Albanese and Mele (cf.~Remark~\ref{rem: power bounded multiplication operators}). Furthermore, from the fact that, for any locally convex space $E$, a continuous linear operator $T\in \mathcal{L}(E)$ is power bounded if and only if $T^2$ is, the next proposition follows immediately.

\begin{proposition}
	Let $\varphi(x)=-x+b$ and $\psi \in \mathscr{O}_M(\mathbb{R})$. The composition operator $C_{\psi,\varphi}$ is power bounded if and only if $((\psi \cdot (\psi\circ\varphi))^n)_{n\in\N}$ is a bounded sequence in $\mathscr{O}_M(\R)$.
\end{proposition}

By the same argument as in the corresponding part of Example \ref{Example: sqrt of 1+x2}, we derive the next proposition.

\begin{proposition}
	\label{constant}
	Let $\varphi:\R\to\R$, $x\mapsto ax+b$, $a,b\in\R, a\neq 0,$ and let $\psi\in C^\infty(\R)$ be a non constant polynomial. Then $C_{\psi,\varphi}$ is not topologizable.
\end{proposition}

\begin{proposition}
	Let $\varphi:\R\to\R$, $x\mapsto x+b$, $b\in\R\setminus\{0\}$ and let $\psi\in C^\infty(\R)$ be a polynomial. Then $C_{\psi,\varphi}$ is power bounded if and only if there is  $c\in\C$, $|c|<1$ such that $\psi(x)=c$ for all $x\in\R$. In this case, the sequence $(C_{\psi,\varphi}^n)$ is convergent to 0 in $\mathcal{L}_b(\mathscr{S}(\R))$.
\end{proposition}
\begin{proof}
	Without loss of generality, we assume $\psi\neq 0$. Since $\psi$ is a polynomial it is of small decay with respect to $\varphi$. Condition (b) of Theorem~\ref{CorollaryPowerBoundednessSmallDecay} (ii) is trivially satisfied.
	
	Assume that $C_{\psi,\varphi}$ is power bounded. By Proposition \ref{constant} $\psi$ has to be constant, $\psi=c$. In case of $|c|=1$ we know that $C_\varphi$ is not power bounded \cite[Remark 2]{FeGaJo18}. Additionally, by condition (a) in Theorem~\ref{CorollaryPowerBoundednessSmallDecay} (ii), for some $q>0$  we have
	$$\infty>\sup_{n\in\N}\sup_{x\in\R}|c|^{n}\frac{(1+|x|)}{(1+|\varphi_n(x)|)^q}\geq  \sup_{n\in\N}|c|^{n}(1+|bn|)$$
	which implies $|c|<1$. Therefore, the sequence $(C_{\psi,\varphi}^n)$ of iterates is clearly pointwise convergent to 0 and because $\mathscr{S}(\R)$ is  a Montel space, we conclude that $(C_{\psi,\varphi}^n)$ is convergent to 0 in  $\mathcal{L}_b(\mathscr{S}(\R))$.
	
	On the other hand, if $\psi$ is constant, $\psi=c$ with $|c|<1$, for $p>0$ we have 
	$$\sup_{n\in\N}\sup_{x\in\R}|c|^{n}\frac{(1+|x|)^p}{(1+|\varphi_n(x)|)^p}=\sup_{n\in\N}\sup_{x\in\R}|c|^n\left(\frac{1+|x|}{1+|x+nb|}\right)^p\leq \sup_{n\in\N}|c|^{n}(1+|bn|)<\infty$$
	so that $C_{\psi,\varphi}$ is power bounded by Theorem~\ref{CorollaryPowerBoundednessSmallDecay}.
\end{proof}





\begin{proposition}
\label{prop:powerbounded for ax+b}
Let $\varphi:\R\to\R$, $x\mapsto ax+b$, $a,b\in\R$, $|a|\notin \{0,1\}$ and let $\psi$ be a non-null polynomial. Then $C_{\psi,\varphi}$ is not power bounded. In case $|a|>1$, $C_{\psi,\varphi}$ is not power bounded for every $\psi\in C^\infty(\R)$ which is of small decay with respect to $\varphi$.
\end{proposition}

\begin{proof}
We first consider the case $|a|>1$. Let  $x_0=\frac{-b}{a-1}$ be the unique fixed point of $\varphi$. We have

$$\frac{|\varphi_n'(x_0)|}{1+|\varphi_n(x_0)|}=\frac{|a|^n}{1+|x_0|}.$$

\noindent Hence, condition (b) from Theorem~\ref{CorollaryPowerBoundednessSmallDecay} (ii) does not hold.

Assume now $|a|<1$ and $\psi$ to be a non-null polynomial. By Proposition \ref{constant} we can reduce to the case $\psi(x)=c$, with $c\in\C\backslash\{0\}$. Since $C_{\psi,\varphi}$ is power bounded if and only if $C_{\psi,\varphi}^2$ is, we can assume  $0<a<1$. Let $p\geq 1$ such that $a^p<|c|$. On account of $\varphi_n(x)=a^nx+b\frac{1-a^n}{1-a}$, for each $q\geq p$, we have
$$\sup_{n\in\N}\sup_{x\in\R}\frac{(1+|x|)^p}{(1+|\varphi_n(x)|)^q}|c|^n\geq \sup_{n\in\N} \frac{\left(1+\frac{1}{a^n}\right)^p}{(1+|\varphi_n\left(\frac{1}{a^n}\right)|)^q}|c|^n\geq \sup_{n\in\N}\left( \frac{|c|}{a^p}\right )^n \frac{1}{\left(1+\left|1+b\frac{1-a^n}{1-a}\right|\right)^q}=\infty.$$
By Theorem~\ref{CorollaryPowerBoundednessSmallDecay}, $C_{\psi,\varphi}$ is not power bounded.  
\end{proof}

\subsection{Power boundedness of $C_{\psi,\varphi}$ with $\varphi$ being a polynomial with $\deg(\varphi)\geq 2$}

The main purpose of the current subsection is to prove Theorem~\ref{main} below which completements Proposition~\ref{prop:powerbounded for ax+b} with a statement for polynomial $\varphi$ of degree larger than one. For this, we first observe that Theorem~\ref{CorollaryPowerBoundednessSmallDecay} and Remark~\ref{powerboundedcomposition} immediately yield the following result.

\begin{proposition}

\label{weightedimplycomposition}
Let $\varphi: \mathbb{R}^d \to \mathbb{R}^d$ be a symbol for $\mathscr{S}(\R^d)$  such that there exists $k,l>0$ such that $|\varphi_n(x)|\geq |x|^k$ when $|x|\geq l, n\in\N$, and let $\psi \in C^{\infty}(\mathbb{R}^d)$ be of small decay with respect to $\varphi$. If $C_{\psi,\varphi}$ acts on $\mathscr{S}(\R^d)$ and is power bounded then $C_{\varphi}$ is also power bounded on  $\mathscr{S}(\R^d)$.

\end{proposition}

\begin{lemma}\label{lemma: auxiliary lemma}
Let $\varphi:\R^d\rightarrow\R^d$ be a symbol for $\mathscr{S}(\R^d)$ such that   $C_\varphi$ is power bounded. Additionally, assume that $|\varphi_n|^2=o(|\varphi_{n+1}|)$ uniformly as $n\rightarrow\infty$, i.e.\ 
$$\forall\varepsilon>0\, \exists N\in\N\,\forall\,n\geq N, x\in \R^d:|\varphi_n(x)|^2\leq \varepsilon|\varphi_{n+1}(x)|.$$
\begin{itemize}
	\item[(i)] There is $k\geq 0$ such that for every $\psi\in\mathscr{O}_M(\R^d)$ and $m\in\N_0$ there are $M_m,q'_m>0$ such that for every $\beta\in\N_0^d$ with $|\beta|\leq m$ it holds
	\begin{equation}\label{eq:key inequality for main 1}
		\sup_{n\in\N}\sup_{|x|\geq k}|(\psi^{n,\varphi})^{(\beta)}(x)|\leq M_m(1+|\varphi_n(x)|)^{q'_m}.
	\end{equation}
	\item[(ii)] Assume that
	\begin{equation}
		\label{eq: additional hypothesis}
		\exists n_0\in\N, c>0\,\forall\,n\geq n_0, x\in\R^d:|\varphi_n(x)|\geq c.
	\end{equation}
	Then, for every $\psi\in\mathscr{O}_M(\R^d)$ and $m\in\N_0$ there are $M_m,q'_m>0$ such that for every $\beta\in\N_0^d$ with $|\beta|\leq m$ it holds
	\begin{equation}\label{eq:key inequality for main 2}
		\sup_{n\geq n_0}\sup_{x\in\R^d}|(\psi^{n,\varphi})^{(\beta)}(x)|\leq M_m(1+|\varphi_n(x)|)^{q'_m}.
	\end{equation}
\end{itemize}
\end{lemma}

\begin{proof}
Since $C_\varphi$ is power bounded, by Remark~\ref{powerboundedcomposition}
\begin{equation}\label{eq: auxiliary inequality 3}
	\exists k\geq 0\,\forall |x|\geq k, n\in\N:\,|\varphi_n(x)|\geq 1.
\end{equation} 

Before we continue, we need to establish the auxiliary estimate \eqref{eq: auxiliary inequality 4} below. For this, we recall that for $\beta\in\N_0^d$ we have
$$ (\psi \circ \varphi_n)^{(\beta)}(x) = \sum_{\substack{\lambda \in \mathbb{N}_0^d \\ 0 \leq |\lambda| \leq |\beta|}} \psi^{(\lambda)}(\varphi_n(x)) \sum_{p(\beta,\lambda)} \beta! \prod_{j=1}^{|\beta|} \frac{(\varphi_n^{(\ell_j)}(x))^{k_j}}{k_j! (\ell_j!)^{|k_j|}}. $$
As $\psi \in \mathscr{O}_M(\mathbb{R}^d)$, for all $\lambda \in \mathbb{N}_0^d$ there exist $C_\lambda,q_\lambda>0$ such that
$$ |\psi^{(\lambda)}(\varphi_n(x))| \leq C_\lambda(1+|\varphi_n(x)|)^{q_\lambda}, \qquad x \in \mathbb{R}^d, \ n \in \mathbb{N}. $$
The power boundedness of $C_{\varphi}$ and Remark~\ref{powerboundedcomposition} (b) thus imply
$$\forall\,\beta\in\N_0^d\,\exists\,C'_\beta, q'_\beta\,\forall n\in\N, x\in\R^d:\, |(\psi \circ \varphi_n)^{(\beta)}(x)|\leq C'_{\beta}(1+|\varphi_n(x)|)^{q'_\beta}.$$
Now, let us fix $m\in\N_0$. Combining the above with \eqref{eq: auxiliary inequality 3}, we obtain the existence of $C_{m}, q_{m}\geq 1$ such that
\begin{equation}\label{eq: auxiliary inequality 4}
	\forall\,\beta\in\N_0^d, |\beta|\leq m, n\in\N, |x|\geq k:\,|(\psi\circ\varphi_n)^{(\beta)}(x)|\leq C_{m}|\varphi_n(x)|^{q_{m}}.
\end{equation}
Next, we fix $n_1\in\N$ large enough such that
$$\forall\,n\geq n_1, x\in\R^d:\, \max\{3, 2^{2m} C_m^2\} |\varphi_{n}(x)|^2 \leq |\varphi_{n+1}(x)|.$$
From the continuity of $C_{\psi,\varphi}^n=C_{\psi^{n,\varphi},\varphi_n}$ and Corollary~\ref{alpha0cor} we obtain the existence of $M_{m}, q'_{m}>0$ with $q_{m}\leq q'_{m}$ such that
\begin{equation}\label{eq: auxiliary inequality 5}
	\forall\,\beta\in\N_0^d, |\beta|\leq m, 1\leq n\leq n_1, |x|\geq k:\,|(\psi^{n,\varphi})^{(\beta)}(x)|\leq M_{m}(1+|\varphi_n(x)|)^{q'_{m}}.
\end{equation}
Therefore, for $\beta\in\N_0^d$ with $|\beta|\leq m$ and $x\in\R^d$ with $|x|\geq k$ we conclude
\begin{eqnarray*}
	|(\psi^{n_1+1,\varphi})^{(\beta)}(x)|&=&|(\psi^{n_1,\varphi}\cdot (\psi\circ\varphi_{n_1}))^{(\beta)}(x)|\\
	&\leq&\sum_{\gamma\leq\beta}{\beta\choose\gamma}\left|\left(\psi^{n_1,\varphi}\right)^{(\gamma)}(x)\right|\cdot\left| \left(\psi\circ\varphi_{n_1}\right)^{(\beta-\gamma)}(x)\right|\\
	&\leq& M_{m}(1+|\varphi_{n_1}(x)|)^{q'_{m}}\sum_{\gamma\leq \beta}{\gamma\choose\beta}C_{m}|\varphi_{n_1}(x)|^{q_{m}}\\
	&=& M_{m} ((1+|\varphi_{n_1}(x)|)^2)^{q'_{m}/2} 2^m C_{m}(|\varphi_{n_1}(x)|^2)^{q_{m}/2}\\
	&\leq& M_{m} (1+3|\varphi_{n_1}(x)|^2)^{q'_m/2} (2^{2m} C^2_{m} |\varphi_{n_1}(x)|^2)^{q_m/2} \\
	&\leq& M_{m}(1+|\varphi_{n_1+1}(x)|)^{q'_{m}/2} |\varphi_{n_1+1}(x)|^{q_{m}/2}\\
	&\leq&M_{m}(1+|\varphi_{n_1+1}(x)|)^{q'_{m}}
\end{eqnarray*}
where we have also used that $2|\varphi_{n_1}(x)| \leq 2|\varphi_{n_1}(x)|^2$ for $|x|\geq k$, and $1 \leq q_{m}\leq q'_{m}$. Thus, \eqref{eq: auxiliary inequality 5} not only holds for $0\leq n\leq n_1$ but also for $n= n_1+1$. Proceeding recursively, we conclude
$$\forall\,\beta\in\N_0^d, |\beta|\leq m, n \in\N_0, |x|\geq k:\,|(\psi^{n,\varphi})^{(\beta)}(x)|\leq M_{m}(1+|\varphi_n(x)|)^{q'_{m}},$$
i.e.\ \eqref{eq:key inequality for main 1} which proves (i).

Assume that \eqref{eq: additional hypothesis} holds. Refering to this additional hypothesis instead of \eqref{eq: auxiliary inequality 3}, the same arguments which led to \eqref{eq: auxiliary inequality 4} yield
\begin{equation}\label{eq: auxiliary inequality 6}
	\forall\,\beta\in\N_0^d, |\beta|\leq m, n\geq n_0, x\in\R^d:\,|(\psi\circ\varphi_n)^{(\beta)}(x)|\leq C_{m}|\varphi_n(x)|^{q_{m}}.
\end{equation}
With the aid of \eqref{eq: auxiliary inequality 6} in place of \eqref{eq: auxiliary inequality 4}, the validity of \eqref{eq:key inequality for main 2} is derived with the same arguments as \eqref{eq:key inequality for main 1}. This proves (ii).
\end{proof}

\begin{remark}\label{remark: additional hypothesis}
Let $\varphi:\R^d\rightarrow\R^d$ be a symbol for $\mathscr{S}(\R^d)$ such that $C_\varphi$ is power bounded. Assume that $|\varphi(x)|> 0$ for each $x\in\R^d$. Then condition \eqref{eq: additional hypothesis} holds. Indeed, by Remark~\ref{powerboundedcomposition} there is $k>0$ such that $|\varphi(x)|\geq k^{1/k}$ whenever $|x|\geq k$. Since $|\varphi|$ has no zeros and $\{x:|x|\leq k\}$ is compact, there is $c>0$ such that $|\varphi(x)|\geq c$, $x\in\R^d$. In particular, $|\varphi|\geq c$ on each of the sets $\varphi_{n-1}(\R)$, $n\in\N$, which shows that \eqref{eq: additional hypothesis} holds true.

Moreover, in case of $d=1$, the additional hypothesis \eqref{eq: additional hypothesis} is automatically satisfied for every polynomial $\varphi:\R\rightarrow\R$ for which $C_\varphi$ is power bounded on $\mathscr{S}(\R)$ without fixed points and $\text{deg}(\varphi)\geq 2$. Indeed, we have either $\varphi(x)>x$ for each $x\in\R$ or $\varphi(x)<x$ for each $x\in\R$. In particular, $\lim_{n\rightarrow\infty}|\varphi_n(x)|=\infty$ for every $x\in\R$. By Remark~\ref{powerboundedcomposition}, there is $k>0$ such that $|\varphi_n(x)|\geq k^{1/k}$ for each $x\in\R\backslash[-k,k]$ and every $n\in\N$. We now consider the case $\varphi(x)>x, x\in\R$. The arguments for the case $\varphi(x)<x, x\in\R$, are mutatis mutandis the same. In particular, $(\varphi_n(x))_{n\in\N}$ is strictly increasing for each $x\in\R$. Additionally, for each $x\in[-k,k]$ there is $n_x\in\N$ such that $\varphi_n(x)>k^{1/k}$ for all $n\geq n_x$. For every $x\in[-k,k]$, let $\delta_x>0$ be such that $\varphi_{n_x}(y)>k^{1/k}$ for every $y\in\R$ with $|x-y|<\delta_x$. Since $[-k,k]$ is compact and the sequences $(\varphi_n(x))_{n\in\N}$, $x\in[-k,k]$, are strictly increasing, there is $n_0\in\N$ such that $\varphi_n(x)\geq k^{1/k}$ for all $x\in[-k,k]$ and $n\geq n_0$. We conclude that \eqref{eq: additional hypothesis} is true.
\end{remark}

Combining Remark~\ref{powerboundedcomposition}, Remark~\ref{sufficientpb}, Lemma~\ref{lemma: auxiliary lemma}, and Remark~\ref{remark: additional hypothesis}, we immediately derive the following result.
\begin{theorem}\label{theorem: additional hypothesis for universal power boundedness}
Let $\varphi:\R^d\rightarrow\R^d$ be a symbol for $\mathscr{S}(\R^d)$ such that $|\varphi_n|^2=o(|\varphi_{n+1}|)$ uniformly as $n\rightarrow\infty$. Moreover, assume that there is $n_0$ such that $|\varphi_{n_0}(x)|> 0$ for each $x\in \R^d$.	
Then, the following are equivalent.
\begin{itemize}
	\item[(i)] $C_\varphi$ is power bounded on $\mathscr{S}(\R^d)$.
	\item[(ii)] $C_{\psi,\varphi}$ is power bounded on $\mathscr{S}(\R^d)$ for every $\psi\in\mathscr{O}_M(\R^d)$ and/or for every smooth $\psi$ of small decay with respect to $\varphi$.
\end{itemize}
\end{theorem}

In order to strengthen the above theorem in case $d=1$ we first prove another auxiliary result.

\begin{proposition}
\label{convergenceiterates}
Let $E$ be a Montel locally convex space and let $T\in \mathcal{L}(E)$.
\begin{itemize}
	\item[(i)] If $\lambda T$ is power bounded for some $\lambda>1$, then $(T^n)_{n\in\N}$ is convergent to 0 in $\mathcal{L}_b(E)$.
	\item[(ii)] If $\left(\frac{1}{n^k}T^n\right)_{n\in\N}$ is bounded in $\mathcal{L}_b(E)$ for some $k>0$, then $\lambda T$ is power bounded for every $\lambda \in (0,1)$.
\end{itemize}
\end{proposition}

\begin{proof}
If $\lambda>1$ and $(\lambda^nT^n(x))_{n\in\N}$ is bounded in $E$ for every $x\in E$ then $(T^n(x))_{n\in\N}$ is convergent to 0 for every $x\in E$, and this is equivalent to $(T^n)_n$ is convergent to 0 in $\mathcal{L}_b(E)$ since $(T^n)_{n\in\N}$ is equicontinuous and $E$ is Montel. This proves (i). To prove (ii), we proceed by contradiction. Let $0<\lambda<1$ and let $x\in E$ and let $p$ be a continuous seminorm on $E$ such that $(\lambda^np(T^n(x)))_{n\in\N}$ is unbounded. For every $k>0$, $\lambda^n<\frac{1}{n^k}$ eventually, hence $\left(\frac{1}{n^k}p(T^n(x))\right)_{n\in\N}$ is also unbounded.
\end{proof}

We can finally prove the main theorem of this subsection which complements Section \ref{subsec:degree 1}. Note that polynomials $\varphi:\R\rightarrow\R$ with $\text{deg}(\varphi)\geq 2$ are symbols for $\mathscr{S}(\R)$ by Remark \ref{rem: one dimensional composition operators}.

\begin{theorem}
\label{main}
Let $\varphi:\R\rightarrow\R$ be a polynomial with $\text{deg}(\varphi)\geq 2$. The following are equivalent.
\begin{itemize}
	\item[(a)]$\varphi$ does not have fixed points (hence, $\varphi$ is of even degree).
	\item[(b)] $C_\varphi$ is power bounded  on $\mathscr{S}(\R)$.
	\item[(c)] $C_\varphi$ is (uniformly) mean ergodic on $\mathscr{S}(\R)$.
	\item[(d)] $\left(C_\varphi^n\right)_{n\in\N}$ converges to $0$ in $\mathcal{L}_b(\mathscr{S}(\R))$.
	\item[(e)] $C_{\varphi}$ is Ces\`aro bounded on $\mathscr{S}(\R)$, i.e.\ the sequence $(\frac{1}{n}\sum_{m=1}^n C_{\varphi}^m)_{n\in\N}$ is bounded in $\mathcal{L}_b(\mathscr{S}(\R))$.
	\item[(f)] $C_{\psi,\varphi}$ is power bounded on $\mathscr{S}(\R)$ for every $\psi \in \mathscr{O}_M(\mathbb{R})$ and/or for every smooth $\psi$ of small decay with respect to $\varphi$.
	\item[(g)] $(C_{\psi,\varphi}^n)_{n\in\N}$ is convergent to 0 in $\mathcal{L}_b(\mathscr{S}(\R))$ for every $\psi\in \mathscr{O}_M(\R)$ and/or for every smooth $\psi$ of small decay with respect to $\varphi$.
	\item[(h)] $(C_{\psi,\varphi}^n)_{n\in\N}$ is (uniformly) mean ergodic on $\mathscr{S}(\R)$ for every $\psi\in \mathscr{O}_M(\R)$ and/or for every smooth $\psi$ of small decay with respect to $\varphi$.
\end{itemize}

\end{theorem}
\begin{proof}
The equivalence among (a), (b), and (c) is \cite[Theorem 3.11]{FeGaJo18} (combined with \cite[Theorem 2.5(b)]{KaSa22}), and (d) is equivalent to (b) due to \cite[Corollary 3.12]{FeGaJo18}. Clearly (b) implies (e).  If we assume (e), then $\left(\frac{1}{n}C^n_{\varphi}\right)$ is bounded in $\mathcal{L}_b(\mathscr{S}(\R))$, and Proposition~\ref{convergenceiterates}(ii) yields that $C_{\psi_0,\varphi}$ is power bounded for $\psi_0(x)=\frac12$. Proposition~\ref{weightedimplycomposition} implies that $C_{\varphi}$ is power bounded and (e) implies (b). The equivalence (b) and (f) holds by Theorem~\ref{theorem: additional hypothesis for universal power boundedness}. Trivially, (g) implies (b), and if (f) is satisfied, then $C_{2\psi,\varphi} = 2C_{\psi,\varphi}$ is power bounded, and (g) follows by Proposition~\ref{convergenceiterates}(i). This shows that (a) to (g) are equivalent. Finally, (f) implies (h) by \cite[Proposition 3.3]{BoPaRi11} and \cite[Theorem 2.5(b)]{KaSa22} while (h) implies (e) by \cite[Theorem 2.5(b)]{KaSa22}.
\end{proof}

\begin{remark}
Besides extending the results about composition operators on $\mathscr{S}(\R)$ from \cite{FeGaJo18} to weighted composition operators, showing that the ergodic properties rely only on the symbol, Theorem~\ref{main} also improves \cite[Theorem 3.11 and Corollary 3.12]{FeGaJo18} for composition operators, by showing that Ces\`aro boundedness and mean ergodicty are equivalent for composition operators. Analyzing the proof, we can even replace the property of Ces\`aro boundedness in condition (e) by the boundedness of $\left(\frac{1}{n^k}C_{\varphi_n}\right)$ in $\mathcal{L}_b(\mathscr{S}(\R))$ for some $k>0$.
\end{remark}

\section{Supercyclicity of weighted composition operators with $\varphi$ being a univariate polynomial.}\label{sec:supercyclicity}

In this short section we use the ideas of Lemma \ref{lemma: auxiliary lemma} in order to prove the next result which complements the results from~\cite{GoPr20}.

\begin{proposition}
Let $\varphi:\R\to\R$ be a polynomial such that $C_{\psi,\varphi}$ is weakly supercyclic for some $\psi\in\mathscr{O}_M(\mathbb{R})$. Then $\varphi(x)=x+b$ with $b\neq 0$.
\end{proposition}

\begin{proof}
If $\psi(x_0)=0$ for some $x_0\in\R$ then $C_{\psi,\varphi}(\mathscr{S}(\R))\subseteq \text{Ker}(\delta_{x_0})$, and $C_{\psi,\varphi}$ cannot be weakly supercyclic. If $\varphi(x_0)=\varphi(x_1)$ for some $x_0\neq x_1$ then $C_{\psi,\varphi}(\mathscr{S}(\R))\subseteq \text{Ker}(\frac{\psi(x_0)}{\psi(x_1)}\delta_{x_1}-\delta_{x_0})$, and again we conclude that $C_{\psi,\varphi}$ is not weakly supercyclic. Hence $ C_{\psi,\varphi}$ is not weakly supercyclic if $\varphi$ has even degree. If $\varphi$ is constant, the range of $C_{\psi,\varphi}$ is the span of $\psi$, and again $ C_{\psi,\varphi}$ is not weakly supercyclic.

When $\varphi(x)=ax+b$ and $a\neq 1$, $a\neq 0$, then for $x_0:=-\frac{b}{a}$, if we denote by $X$ the linear span of $\{\delta_{x_0},\delta_{x_0}^{(1)}\}\subseteq \mathscr{S}(\R)'$, satisfies $C_{\psi,\varphi}'(X)\subseteq X$, and the matrix representing the restriction of $C_{\psi,\varphi}'$ to $X$ in the basis $\{\delta_{x_0},\delta_{x_0}^{(1)}\}$ is 
$$\left(\begin{array}{cc} \psi(x_0)& \psi'(x_0)\\ 0&a\psi(x_0)\end{array}\right).$$

\noindent Since $a\neq 1$ we conclude that $C_{\psi,\varphi}'$ has two eigenvalues, and then $C_{\psi,\varphi}$ is not weakly supercyclic by \cite[Proposition I.26]{BM}.

When $\varphi(x)=x$, $C_{\psi,\varphi}$ is a multiplication operator with $\{\delta_x:\ x\in \R\}$ being eigenvectors of $C_{\psi,\varphi}'$, and again we conclude by \cite[Proposition I.26]{BM} that $C_{\psi,\varphi}$ cannot be weakly supercyclic. If $\varphi$ has two fixed points again the same argument applies.

To finish, we have to conclude that $C_{\psi,\varphi}$ is not weakly supercyclic when $\varphi$ is a polynomial of odd degree bigger or equal than 3 with only one fixed point $a$ and $\psi(x)\neq 0$ for any $x\in\R$. We proceed by contradiction. Since non null multiples of weakly supercyclic operators are weakly supercyclic, we can assume $\psi(x_0)=1$. Under these hypotheses, there is $k>|a|$ such that $|x|\geq k$ implies $|\varphi(x)|>|x|$ and $|\varphi_n(x)|^2<|\varphi_{n+1}(x)|$ for all $n\in\N$. Let $M,q>0$ such that $|\psi(x)| \leq M|x|^q$ for $|x|>k$. We argue as in Lemma~\ref{lemma: auxiliary lemma} to get $|\psi^{n,\varphi}(x)|\leq M|\varphi_n(x)|^q$ for all $|x|>k$. From \cite[Proposition I.26]{BM} we get that $C_{\psi,\varphi}$ restricted to $\text{Ker}(\delta_a)$ is weakly hypercyclic, and, for any $|x|>k$ and $f\in \text{Ker}(\delta_a)$, being a hypercyclic vector 
$$ |C_{\psi,\varphi}^nf(x)|=|\psi^{n,\varphi}(x)|\cdot |f(\varphi_n(x))|\leq M|\varphi_n(x)|^q|f(\varphi_n(x))|\leq  M\sup_{x\in\R}|x|^q|f(x)|, $$
and this supremum is finite since $f\in \mathscr{S}(\R)$, which is a contradiction with the assumption that $f$ is a weakly hypercyclic vector.
\end{proof}

For the translation operator, i.e.\ $C_\varphi$ with $\varphi(x)=x+1$, we refer the reader to~\cite{GoPr20} for sufficient conditions on $\psi \in \mathscr{O}_M(\mathbb{R})$ for $C_{\psi,\varphi}$ to be weakly supercyclic on $\mathscr{S}(\R)$ as well as for other linear dynamical properties for these operators.

\quad

\noindent\textbf{Acknowledgement.} The research on the topic of this article was initiated during a visit of the third named author at  Departamento de Matem\'atica Aplicada, E. Polit\'{e}cnica Superior de Alcoy, Universidad Polit\'ecnica de Valencia. He is very grateful to his colleagues from Valencia for the cordial hospitality during this stay as well as during numerous earlier ones. The first named author was supported by GV PROMETEU/2021/070. The third named author was partially supported by the project PID2020-119457GB-100 funded by MCIN/AEI/10.13039/501100011033 and by “ERDF A way of making Europe”.

\end{document}